\pdfoutput=1
\documentclass[reqno,a4paper]{amsart}

\usepackage[english,activeacute]{babel}
\usepackage{amssymb,amsmath,amsthm,amsfonts,mathrsfs,scalefnt,graphicx}
\usepackage[bookmarks=false]{hyperref}
\usepackage[font=footnotesize]{caption}
\usepackage{lmodern}
\usepackage{textcomp}
\usepackage{ifpdf}
\usepackage{pst-all}
\usepackage{tabularx}
\usepackage[T1]{fontenc}
\usepackage{url}  
\usepackage[latin1]{inputenc}
\usepackage[foot]{amsaddr}
\theoremstyle{plain}
\newtheorem{theorem}{Theorem}[section]
\newtheorem{lemma}[theorem]{Lemma}
\newtheorem{proposition}[theorem]{Proposition}
\newtheorem{corollary}[theorem]{Corollary}

\newtheorem{remark}[theorem]{Remark}
\numberwithin{equation}{section}

\newcommand{\Di}  {{\mathbb{D}_\infty}}
\newcommand{\Dt}[1] {{\mathbb{D}_{#1}}}
\newcommand{\Ct}[1] {{\mathbb{C}_{#1}^{\uparrow}}}

\newcommand{\Db}  {{\mathbb D}}
\newcommand{\Nb}  {{\mathbb N}}
\newcommand{\Rb}  {{\mathbb R}}

\newcommand{\Fb}  {{\mathbb F}}

\newcommand{\As} {{\mathcal A}}

\newcommand{\Cs} {{\mathcal C}}

\newcommand{\Fs} {{\mathcal F}}
\newcommand{\Gs} {{\mathcal G}}

\newcommand{\Is} {{\mathcal I}}

\newcommand{\Ls} {{\mathcal L}}

\newcommand{\Ps} {{\mathcal P}}

\newcommand{\Rs} {{\mathcal R}}

\title[Common ancestor type distribution]{Common ancestor type distribution: A Moran model and its deterministic limit}
\author[F.~Cordero]{Fernando Cordero$^1$}
\address{$^1$Faculty of Technology, University of Bielefeld, Universit\"{a}tsstrasse 25, 33615 Bielefeld, Germany. E-mail: fcordero@techfak.uni-bielefeld.de}

\date{\today}%

\begin{document}
\begin{abstract}
We study the common ancestor type distribution in a $2$-type Moran model with population size $N$, mutation and selection, and in the deterministic limit regime arising in the former when $N$ tends to infinity, without any rescaling of parameters or time. In the finite case, we express the common ancestor type distribution as a weighted sum of combinatorial terms, and we show that the latter converges to an explicit function. Next, we recover the previous results through pruning of the ancestral selection graph (ASG). The notions of relevant ASG, finite and asymptotic pruned lookdown ASG permit to achieve this task.
\vspace{.2cm}\\
\textbf{Keywords:} Common ancestor type distribution; ancestral selection graph; lookdown graph; pruning; Moran model with selection and mutation\\
\textbf{Mathematics Subject Classification (2010):} Primary 92D25, 60G09, 60J28; Secondary 60J27, 60J35
\end{abstract}

\maketitle
\section{Introduction}
A variety of (stochastic) models describes the interplay of mutation and selection in populations in the forward direction of time, of which the Wright-Fisher and Moran models appear as major cornerstones (see, e.g., \cite{Du08, Eth}). Corresponding  ancestral processes, starting at present and tracing back the ancestry of individuals into the past, are well studied and currently constitute an active area of research (\cite{KroNe97,NeKro97,Popf13,LKBW15}). It is a common feature of most of these models that at any time there is one individual which is at a later time ancestral to the whole population. Such an individual is called common ancestor. 

In this paper we are interested in the type distribution of the common ancestor in a $2$-type Moran model with mutation and selection and its asymptotic behaviour when the population size tends to infinity. For the latter we consider two different regimes: the diffusion limit and the deterministic limit. In the first regime, the time and the parameters of mutation and selection are rescaled in such a way that, when the population size tends to infinity, the proportion of fit individuals converges to the Wright-Fisher diffusion. By contrast, when the time and the parameters of the model are not rescaled and the population size tends to infinity, the proportion of fit individuals converges to the solution of an ordinary differential equation (see \cite{Co15}). In this case, we talk about the deterministic limit regime. 

In the diffusion limit regime, the common ancestor type distribution has been widely studied (see \cite{Ki62,Ma09,Popf13} for the case without mutation and \cite{Fe02,Ta07,LKBW15} for the general case). Let $h_\infty(x)$ be the probability that the common ancestor is fit given that the current proportion of fit individuals is $x$. In \cite{Ta07}, Taylor shows that $h_\infty$ is the solution of a boundary value problem. In addition, he gives a series expansion for $h_\infty$ in terms of Fearnhead's coefficients (introduced in \cite{Fe02}). In the recent work \cite{LKBW15}, the authors construct a pruned version of the untyped ancestral selection graph, called pruned lookdown ancestral selection graph (pruned LD-ASG). Based on the pruned LDS-ASG, they recover the series expansion for $h_\infty$ in a graphical way. They also show that the Fearnhead's coefficients correspond to the tail probabilities of the stationary number of lines in the LD-ASG. The main goal of this paper is to extend these results to the finite and the deterministic limit setting.

In the Moran model of size $N$ with selection and mutation part of the aforementioned results are also available. Let us denote by $h_k^N$ the probability that the common ancestor is fit given that the initial population has exactly $k$ fit individuals. In \cite{KHB13}, $h_k^N$ is expressed as a weighted finite sum of combinatorial terms. The weights are defined through a $2$-step forward recursion, and we refer to them as Fearnhead-type coefficients. This representation of $h_k^N$ is unfortunately not closed, since one of the equations in the recursion depends on the value $h_{N-1}^N$. In order to complete the picture we provide an analytical and a graphical approach. 

The analytical approach consists of two main steps. We first characterise the Fearnhead-type coefficients as the unique solution of a slightly different recursion depending only on the parameters of selection and mutation. Next, using some elements of the theory of matrices, we show that the Fearnhead-type coefficients correspond to the tail probabilities of some random variable. This approach does not provide any extra graphical meaning beyond the results in \cite{KHB13}. 

The graphical approach permits to characterise the $h_k^N$ as in \cite{KHB13}, and simultaneously, to recover the results we obtained with our first approach. 
In the case without mutation, this can be done as in \cite{Popf13} with the help of the ancestral selection graph (ASG - \cite{KroNe97,NeKro97}). In the presence of mutations part of the ASG becomes irrelevant, and hence, the problem can not be treated in the same way. We solve the problem by means of two ways of pruning the untyped ASG. We define the relevant ancestral selection graph (relevant ASG), and then, following the lines of \cite{LKBW15}, we extend the notion of pruned LD-ASG to the finite population case. Using these constructions, we show that the Fearnhead-type coefficients correspond to the tail probabilities of, on one hand, the asymptotic number of lines in the relevant ASG, and on the other hand, the stationary number of lines in the LD-ASG. Both representations provide a graphical explanation to the representation of the probabilities $h_k^N$ given in \cite{KHB13}. The pruned LD-ASG gives, in addition, a probabilistic interpretation of the Fearnhead-type recursion.

In the deterministic limit framework, we also use two approaches. First, we show that the Fearnhead-type coefficients converge to the tail distribution of a geometric random variable. We deduce that the probabilities $h_k^N$ converge to a function $h$ which is explicitly computed. In order to provide a graphical interpretation, we take an approach similar to the one used in the finite case. The main difficulty is that coalescence events are absent in any suitable asymptotic version of the ancestral selection graph. Therefore, the notion of common ancestor does not make sense anymore. However, the convergence properties of the number of lines in the finite pruned LD-ASG give us a way to define an asymptotic version of the pruned LD-ASG (in the deterministic limit regime). This new object and the notion of representative ancestral type lead to a nice graphical interpretation of the function $h$.

The paper is organized as follows. In Section \ref{s2}, we give a short description of the $2$-type Moran model with selection and mutation, the diffusion limit, the deterministic limit and some known facts about the common ancestor type distribution. Section \ref{sapv} contains the analytical approach to the study of the probabilities $h_k^N$ and their asymptotic behaviour in the deterministic limit setting. Sections \ref{s5} and \ref{s6} are devoted to give graphical interpretations to the problems studied in Section \ref{sapv}. In Section \ref{s5}, we treat the finite case. We first recall the notion of ASG and we introduce the relevant ASG and the lookdown ASG. These objects are used to obtain the desired graphical interpretations. In Section \ref{s6}, we give a meaning to the asymptotic results obtained in Section 4. The paper ends with Appendices \ref{A1} and \ref{bm}. In Appendix \ref{A1}, we provide some technical results about the Skorohod topology which are needed in Section \ref{s6}. Finally, in Appendix \ref{bm} we compare the deterministic limit, with the asymptotic properties of a related $2$-type branching model.

All along the paper, we use the following notation, for $k,m\in\Nb_0$, with $k<m$, $[m]_k:=[k,m]\cap\Nb_0$. When $k=1$, we simply write $[m]$ instead of $[m]_1$.

\section{Preliminaries}\label{s2}
\subsection{\texorpdfstring{The $2$-type Moran model with selection and mutation}{}}\label{pre}
We consider a population of size $N$ in which each individual is characterised by a type $i\in\{0,1\}$. If an individual reproduces, its single offspring  inherits the parent's type and replaces a uniformly chosen individual, possibly its own parent. The replaced individual dies, keeping the size of the population constant. 

Individuals of type $1$ reproduce at rate $1$, whereas individuals of type $0$ reproduce at rate $1+s_N$, $s_N\geq0$. Mutation occurs independently of reproduction. An individual of type $i$ mutates to type $j$ at rate $u_N\,\nu_j$, $u_N\geq 0$, $\nu_j\in(0,1)$, $\nu_0+\nu_1=1$.

The Moran model has a well-known graphical representation as an interacting particle system (see Fig. \ref{fig1so}). Individuals are represented by horizontal lines. Time runs from left to right. Each reproduction event is represented by an arrow with the parent at its tail and the offspring at its head. We decompose reproductions into two kinds of events: neutral and selective. Neutral reproductions are depicted as arrows with filled heads and selective ones as arrows with open heads. Neutral arrows may be used by all individuals, whereas selective arrows may be only used by individuals of type $0$. Mutations to type $0$ are represented by open circles and mutations to type $1$ by filled circles.

More precisely, for each $i,j\in [N]$ with $i\neq j$, $\lambda_{i,j}^{\vartriangle,N}$ and $\lambda_{i,j}^{\blacktriangle,N}$ denote two Poisson processes with respective rates $s_N/N$ and $1/N$. Similarly, for each $i\in [N]$, $\lambda_{i}^{0,N}$ and $\lambda_{i,j}^{1,N}$ stand for two Poisson processes with respective rates $u_N\nu_0$ and $u_N\nu_1$. We assume that all these processes are independent and we refer 
$$\Lambda^N:=\{\lambda_{i}^{0,N},\lambda_{i}^{1,N},\{\lambda_{i,j}^{\vartriangle,N},\lambda_{i,j}^{\blacktriangle,N}\}_{j\in[N]/\{i\}} \}_{i\in[N]}$$ 
as the \textit{reproduction-mutation process}. Now, we draw arrows and circles in the space-time coordinate system $[0,\infty)\times[N]$ as follows. At the arrival times of  $\lambda_{i,j}^{\vartriangle,N}$ and $\lambda_{i,j}^{\blacktriangle,N}$, we draw selective or neutral arrows respectively, going from line $i$ to $j$. At the arrival times of $\lambda_{i}^{0,N}$ and $\lambda_{i,j}^{1,N}$, we draw respectively open or filled circles at line $i$. So far we have constructed an untyped version of the Moran model. Finally, given an initial configuration of types, we propagate the types forward in time respecting reproduction and mutation events.
\ifpdf
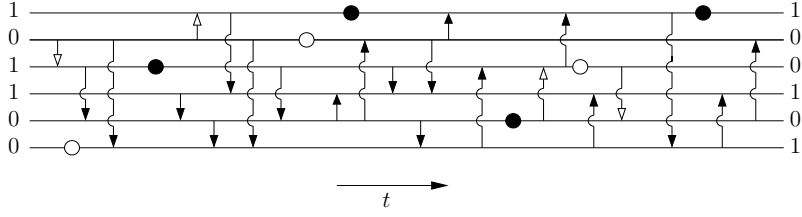
\begin{figure}[!ht]
\begin{picture}(0,0)%
\centerline{\includegraphics[width=10.5cm, height=2.8cm]{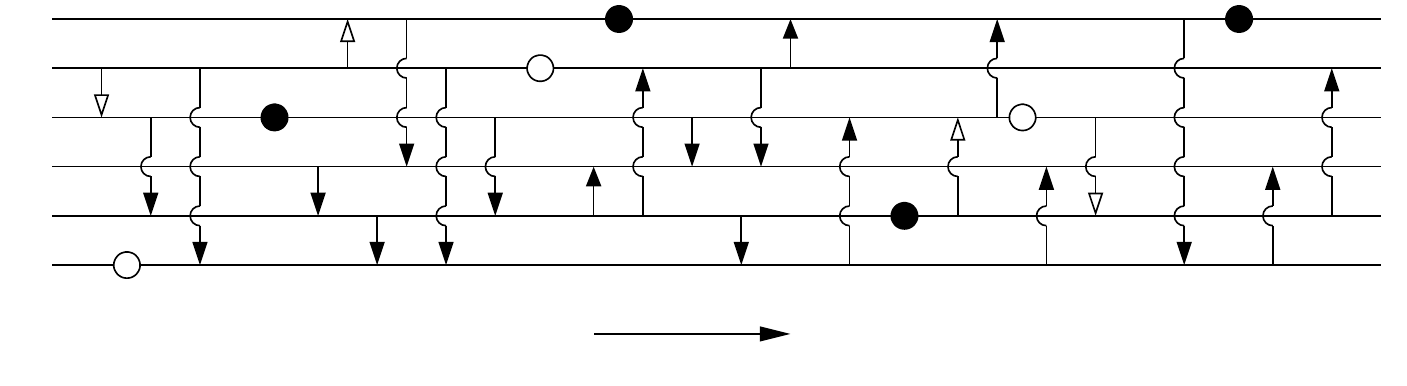}}
\end{picture}%
\setlength{\unitlength}{4144sp}%
\begingroup\makeatletter\ifx\SetFigFont\undefined%
\gdef\SetFigFont#1#2#3#4#5{%
  \reset@font\fontsize{#1}{#2pt}%
  \fontfamily{#3}\fontseries{#4}\fontshape{#5}%
  \selectfont}%
\fi\endgroup%
\centerline{\resizebox*{10.5cm}{2.8cm}{
\begin{picture}(6420,1753)(-14,-1061)
\put(3061,-1006){\makebox(0,0)[lb]{\smash{{\SetFigFont{12}{14.4}{\rmdefault}{\mddefault}{\updefault}{\color[rgb]{0,0,0}$t$}%
}}}}
\put(  1,344){\makebox(0,0)[lb]{\smash{{\SetFigFont{12}{14.4}{\rmdefault}{\mddefault}{\updefault}{\color[rgb]{0,0,0}$0$}%
}}}}
\put(  1,569){\makebox(0,0)[lb]{\smash{{\SetFigFont{12}{14.4}{\rmdefault}{\mddefault}{\updefault}{\color[rgb]{0,0,0}$1$}%
}}}}
\put(  1,-106){\makebox(0,0)[lb]{\smash{{\SetFigFont{12}{14.4}{\rmdefault}{\mddefault}{\updefault}{\color[rgb]{0,0,0}$1$}%
}}}}
\put(  1,-331){\makebox(0,0)[lb]{\smash{{\SetFigFont{12}{14.4}{\rmdefault}{\mddefault}{\updefault}{\color[rgb]{0,0,0}$0$}%
}}}}
\put(  1,-556){\makebox(0,0)[lb]{\smash{{\SetFigFont{12}{14.4}{\rmdefault}{\mddefault}{\updefault}{\color[rgb]{0,0,0}$0$}%
}}}}
\put(6391,-556){\makebox(0,0)[lb]{\smash{{\SetFigFont{12}{14.4}{\rmdefault}{\mddefault}{\updefault}{\color[rgb]{0,0,0}$1$}%
}}}}
\put(6391,-331){\makebox(0,0)[lb]{\smash{{\SetFigFont{12}{14.4}{\rmdefault}{\mddefault}{\updefault}{\color[rgb]{0,0,0}$0$}%
}}}}
\put(6391,-106){\makebox(0,0)[lb]{\smash{{\SetFigFont{12}{14.4}{\rmdefault}{\mddefault}{\updefault}{\color[rgb]{0,0,0}$1$}%
}}}}
\put(6391,119){\makebox(0,0)[lb]{\smash{{\SetFigFont{12}{14.4}{\rmdefault}{\mddefault}{\updefault}{\color[rgb]{0,0,0}$0$}%
}}}}
\put(6391,344){\makebox(0,0)[lb]{\smash{{\SetFigFont{12}{14.4}{\rmdefault}{\mddefault}{\updefault}{\color[rgb]{0,0,0}$0$}%
}}}}
\put(6391,569){\makebox(0,0)[lb]{\smash{{\SetFigFont{12}{14.4}{\rmdefault}{\mddefault}{\updefault}{\color[rgb]{0,0,0}$1$}%
}}}}
\put(  1,119){\makebox(0,0)[lb]{\smash{{\SetFigFont{12}{14.4}{\rmdefault}{\mddefault}{\updefault}{\color[rgb]{0,0,0}$1$}%
}}}}
\end{picture}
}}
\caption{\label{fig1so} Graphical illustration of the Moran model: types are indicated for the initial configuration (left) and the final one (right). Filled  and open circles represent mutations to type 1 and 0 respectively. Neutral reproductions are depicted as arrows with filled heads and selective ones as arrows with open heads.}
\end{figure}
\fi
Relevant information of the evolution of the population is given by the continuous-time Markov chain $X^N:=(X_t^N)_{t\geq 0}$, where $X_t^N$ counts the number of individuals of type $0$ at time $t$. The infinitesimal parameters of $X^N$ are given by
\begin{equation*}
 q_{k,k+\ell}^N:=\left\{ \begin{array}{ll}
                             k\,\frac{(N-k)}{N}\,(1+s_N) +(N-k)\,u_N\,\nu_0& \textrm{if $\ell=1$},\\
                            &\\
                             k\,\frac{(N-k)}{N} +k \,u_N\,\nu_1& \textrm{if $\ell=-1$},\\
                         \end{array}\right.
\end{equation*}
and $q_{k,k+\ell}^{N}:=0$ for $|\ell|>1$.
Equivalently, the infinitesimal generator of $X^N$, denoted by $\As_{X^N}$, is given by
\begin{equation}\label{igen1}
\As_{X^N}f(k):=q_{k,k+1}^N\,\left(f(k+1)-f(k)\right)
+q_{k,k-1}^N\,\left(f(k-1)-f(k)\right).
\end{equation}
In other words, $X^N$ is a birth-death process with birth rates $\lambda_k^N:=q_{k,k+1}^N$ and death rates $\mu_k^N:=q_{k,k-1}^N$.
In particular, when $u_N>0$, $X^N$ admits a unique stationary distribution, which is given by
\begin{equation}\label{sd}
\pi_{X^N}(k):=C_N\prod_{i=1}^k \frac{\lambda_{i-1}^N}{\mu_i^N},\quad k\in[N]_0,
\end{equation}
where $C_N$ is a normalising constant (see \cite{Du08}). When $u_N=0$, by contrast, $X^N$ is an absorbing Markov chain with $0$ and $N$ as absorbing states.
\subsection{The diffusion limit}
A natural diffusion limit arises in the Moran model when the parameters of selection and mutation satisfy
\begin{equation}\label{dlas}
 \lim_{N\rightarrow\infty}Nu_N=\theta\in(0,\infty)\quad\textrm{and}\quad\lim_{N\rightarrow\infty}Ns_N=\sigma\in(0,\infty).
\end{equation}
We briefly recall here its construction (see also \cite[p. 71, Lemma 5.11]{Eth}). Let $Y_t^N$ denote the proportion of fit individuals in the population at time $Nt$, i.e. 
$$Y_t^N:=\frac{1}{N}X_{Nt}^N,\quad t\geq 0.$$
From \eqref{igen1}, the infinitesimal generator of the process $Y^N:=(Y_t^N)_{t\geq 0}$ is given by
\begin{align*}
\As_{Y^N}f(p)=&N^2 p(1-p)\left(f\left(p+N^{-1}\right)+f\left(p-N^{-1}\right)-2f(p)\right)\nonumber\\
&+N^2\left[ s_N\, p(1-p) +(1-p)\,u_N\,\nu_0\right]\left(f\left(p+N^{-1}\right)-f(p)\right)\nonumber\\
&+N^2p\,u_N\,\nu_1\left(f\left(p-N^{-1}\right)-f(p)\right),
\end{align*}
for every $p\in E_N:=\left\{i/N:\,i\in[N]_0\right\}$. Therefore, if $f\in\Cs^3([0,1])$, an appropriate Taylor expansion leads to
$$\lim\limits_{N\rightarrow\infty}\sup\limits_{p\in E_N} |\As_{Y^N}f(p)-\As_Y f(p)|=0,$$
where $\As_Y$ is the generator of the Wright-Fisher diffusion $Y$, i.e.
$$\As_Y f(x):=x(1-x)\frac{d^2f}{dx^2}(x)+ \left[\sigma(1-x)x+\theta\nu_0(1-x)-\theta \nu_1x\right]\frac{df}{dx}(x),\quad x\in[0,1].$$
Consequently, $Y^N\xrightarrow[N\rightarrow\infty]{(d)}Y$ (see \cite[Theorems 1.6.1, 4.2.11 and 8.2.1]{Ku86}).
\subsection{The deterministic limit}\label{s3}
In contrast to the diffusion limit framework, where the parameters of the model satisfy \eqref{dlas}, we consider all along this paper constant parameters of mutation and selection, i.e. $u_N=u\geq 0$ and $s_N=s\geq 0$. In addition, we do not rescale the time. In \cite{Co15} it is shown that a deterministic limit emerges when the size of the population converges to infinity. We summarize here the convergence results related to this deterministic limit. Let $Z^N:=(Z_t^N)_{t\geq 0}$ be the continuous-time Markov chain given by
$Z_t^N:=\frac{1}{N}X_t^N$, $t\geq 0.$
For each $z_0\in[0,1]$, we denote by $z(z_0,\cdot)$ the solution of the following ordinary differential equation
\begin{equation}\label{deteq}
\frac{dz}{dt}(t)=sz(t)\left(1-z(t)\right)+u\nu_0\left(1-z(t)\right)-u\nu_1 z(t),\quad t\geq 0.
\end{equation} 
Eq. \eqref{deteq} has a unique stable point which is given by
\begin{equation}\label{xo+-}
x_0^+:=\left\{ \begin{array}{ll}
                         \frac{s-u+\sqrt{\Delta}}{2s}  & \textrm{if $s>0$},\\
                             
                          \nu_0 & \textrm{if $s=0$},
                         \end{array}\right.
\end{equation}
where $\Delta:=(s-u)^2+4su\nu_0$. In addition, $x_0^+$ satisfies
 \begin{equation}\label{at}
 \lim\limits_{t\rightarrow\infty}z(z_0,t)=x_0^+.
\end{equation}
It is shown in \cite[Proposition 3.1]{Co15} that
\begin{equation*}
\lim_{N\rightarrow\infty}Z_0^N=z_0\in[0,1]\Rightarrow\forall\,\varepsilon, T>0:\quad\lim\limits_{N\rightarrow\infty} P\left(\sup\limits_{t\leq T}|Z_t^N-z(z_0,t)|>\varepsilon\right)=0.
\end{equation*}

Now, we set $g(x):= -(2+s)x^2+ \left(2+s-u(\nu_0-\nu_1)\right)\,x+ u\nu_0$, $x\in\Rb$, and we define the Gaussian diffusion $V^{z_0}:=(V_t^{z_0})_{t\geq 0}$ by
$$V_t^{z_0}:=\left\{ \begin{array}{ll}
                           F(z(z_0,t))\int\limits_0^t \frac{\sqrt{g(z(z_0,v))}}{F(z(z_0,v))}\,dB_v  & \textrm{if $z_0\neq x_0^+$},\\
                             
                           \sqrt{g(x_0^+)}\, e^{-\sqrt{\Delta}\,t}\int\limits_0^t e^{\sqrt{\Delta}\,v}\,dB_v  & \textrm{if $z_0=x_0^+$},
                         \end{array}\right.$$
where $(B_t)_{t\geq 0}$ is a standard Brownian motion. In addition, we introduce the characteristic functions
$\psi^{N}(t,\theta):=E\left[e^{i\theta\sqrt{N}\left(Z_t^N-z(z_0,t)\right)}\right]$ and $\psi(t,\theta):=E\left[e^{i\theta V_t^{z_0}}\right]$. Assuming that $u>0$, \cite[Theorem 3.4]{Co15} tell us that
\begin{equation*}
\lim_{N\rightarrow\infty}\sqrt{N}(Z_0^N-z_0)=0\Rightarrow\lim\limits_{N\rightarrow\infty}\sup_{t\geq 0}|\psi_N(t,\theta)-\psi(t,\theta)|=0.
 \end{equation*}
As a consequence, it is deduced that
\begin{equation}\label{csd}
\pi_{Z^N}\xrightarrow[N\rightarrow\infty]{w}\delta_{x_0^+}
\end{equation}
where $\pi_{Z^N}$ denotes the stationary distribution of $Z^N$ (see \cite[Corollary 3.6]{Co15}).
\begin{remark}
Eq. \eqref{deteq} describes the evolution of the proportion of fit individuals in the well-known $2$-allele paralell mutation-selection (see \cite[p. 265]{CK70}).
\end{remark}
\begin{remark}
The results presented here were obtained in \cite{Co15} with the help of classical results for density dependent families of Markov chains (see \cite{Ku76,Ku81,Ku86}).  
\end{remark}
\subsection{The common ancestor type distribution: known facts}\label{catkf}
It has been shown in \cite{Ta07} that, in the diffusion limit framework, the common ancestor type distribution takes the form
\begin{equation}\label{dilh}
h_\infty(x)=\sum\limits_{n=0}^\infty \alpha_n\, x(1-x)^n,
\end{equation}
where the coefficients $(\alpha_n)_{n\geq0}$ satisfy the following second-order recursion
\begin{equation}
 (n+\theta \nu_1)\alpha_n-(n+\sigma+\theta)\alpha_{n-1}+\sigma \alpha_{n-2}=0,\quad n\geq 2,
\end{equation}
with boundary conditions $\alpha_0=0$ and $\lim_{n\rightarrow\infty} \alpha_{n+1}/\alpha_{n}=0$. The $(\alpha_n)_{n\geq0}$ were introduced in \cite{Fe02} and we refer to them as \textit{Fearnhead's coefficients}. In the recent work \cite{LKBW15}, the authors show that $\alpha_n=P(L_\infty>n)$, where $L_\infty$ denotes the stationary number of lines in the lookdown ancestral selection graph. In addition, a graphical proof of \eqref{dilh} is provided. 

In the finite population case, a representation of the common ancestor type distribution, similar to \eqref{dilh}, is given in \cite{KHB13}. More precisely, a first-step analysis applied to the probabilities $h_k^N$ leads to
\begin{equation}\label{hkak}
 h_k^N=\frac{k}{N}\sum\limits_{n=0}^{N-k} a_n^{N} \prod\limits_{j=0}^{n-1}\frac{N-k-j}{N-1-j},\quad k\in[N]_0,
\end{equation}
where the coefficients $a_n^N$ satisfy, for $n\in[N-1]_2$,
\begin{align}
a_0^N&=1,\label{a0}\tag{$E_0$}\\
a_1^N&=1-N(1-h_{N-1}^N),\label{a1}\tag{$E_1$}\\
\left(\frac{n}{N}+u\nu_1\right)a_n^N&=\left(\frac{n}{N}+\frac{N-(n-1)}{N}s+u\right)a_{n-1}^N-\frac{N-(n-1)}{N}sa_{n-2}^N.\label{an}\tag{$E_n$}
\end{align}

\section{The common ancestor type distribution: an analytical approach}\label{sapv}
As in the diffusion limit setting, in the $2$-type Moran model of size $N$ subject to selection and mutation, at any time $t$, there is a unique individual that is, at some later time $v>t$, ancestral to the whole population (this result follows in the diffusion limit from \cite[Theorem 3.2]{KroNe97}). To see this, 
we fix $t\geq 0$ and we define the \textit{offspring-type process} $\Theta^{t,N}:=(\{\theta_i^{t,N}(v)\}_{i\in [N]},\{j_i^{t,N}(v)\}_{i\in [N]})_{v\geq 0}$ as follows. \\
\textbullet For $i\in [N]$, $\theta_i^{t,N}(0):=\{i\}$ and $j_i^{t,N}(0)$ denotes the type of the individual at line $i$ at time $t$.\\
\textbullet For $v>0$ and $i\in [N]$, $\theta_i^{t,N}(v)$ holds the set of lines occupied at time $t+v$ by descendants of the individual located at line $i$ at time $t$. We denote by $j_i^{t,N}(v)$ the type of the individual placed at line $i$ at time $t+v$.

The offspring-type process is a continuous-time Markov chain with state space
$$\varSigma_N:=\Ps_*([N])\times\{0,1\}^N,$$
where $\Ps_*([N]):=\{\{A_i\}_{i\in [N]}: \forall i\neq j,\, A_i\subset [N],\, A_i\bigcap A_j=\emptyset,\,\mathop{\bigcup}_{k\in [N]} A_k=[N]\}.$
We point out that $\Theta^{t,N}$ can be constructed using the reproduction-mutation process $\Lambda^N$ defined in Section \ref{pre}, or by exhibiting its transition probabilities. 
The set 
$$\Xi_N:=\{(\{A_i\}_{i\in [N]},\{j_i\}_{i\in [N]})\in \varSigma_N: \exists j\in [N],\, A_j=[N]\},$$ 
is a closed set of $\Theta^{t,N}$. From any state outside of $\Xi_N$, $\Theta^{t,N}$ reaches $\Xi_N$ with positive probability and the state space is finite. Therefore, the probability of absorption in $\Xi_N$ is equal to one. This means that the offspring of one of the individuals at time $t$ will fix at a later time. Such individual is called \textit{the common ancestor} at time $t$ and we denote its type by $I_t^N$. The lineage of these individuals over time defines the so-called \textit{ancestral line} (see Fig. \ref{fig2so}). 
\ifpdf
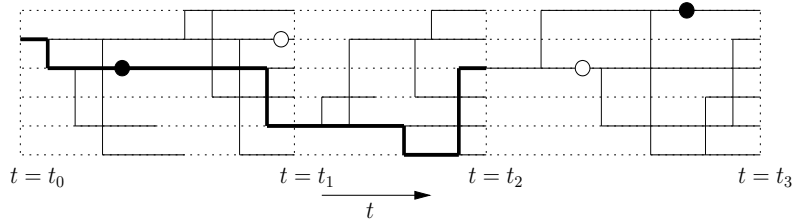
\begin{figure}[!ht]
\begin{picture}(0,0)%
\centerline{\resizebox*{10cm}{3cm}{\includegraphics{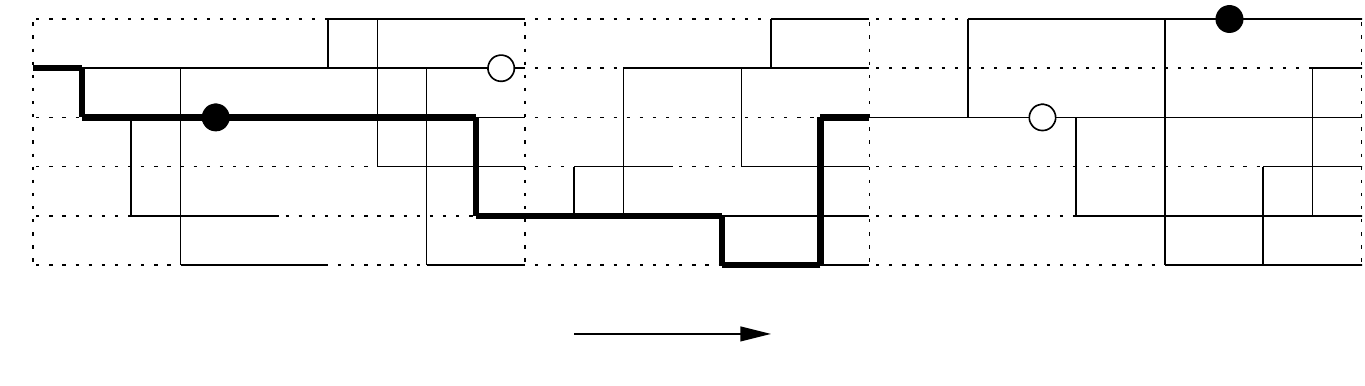}}}
\end{picture}%
\setlength{\unitlength}{4144sp}%
\begingroup\makeatletter\ifx\SetFigFont\undefined%
\gdef\SetFigFont#1#2#3#4#5{%
  \reset@font\fontsize{#1}{#2pt}%
  \fontfamily{#3}\fontseries{#4}\fontshape{#5}%
  \selectfont}%
\fi\endgroup%
\centerline{\resizebox*{10cm}{3cm}{
\begin{picture}(6237,1747)(76,-1061)
\put(3916,-736){\makebox(0,0)[lb]{\smash{{\SetFigFont{12}{14.4}{\rmdefault}{\mddefault}{\updefault}{\color[rgb]{0,0,0}$t=t_2$}%
}}}}
\put(6166,-736){\makebox(0,0)[lb]{\smash{{\SetFigFont{12}{14.4}{\rmdefault}{\mddefault}{\updefault}{\color[rgb]{0,0,0}$t=t_3$}%
}}}}
\put(3061,-1006){\makebox(0,0)[lb]{\smash{{\SetFigFont{12}{14.4}{\rmdefault}{\mddefault}{\updefault}{\color[rgb]{0,0,0}$t$}%
}}}}
\put( 91,-736){\makebox(0,0)[lb]{\smash{{\SetFigFont{12}{14.4}{\rmdefault}{\mddefault}{\updefault}{\color[rgb]{0,0,0}$t=t_0$}%
}}}}
\put(2341,-736){\makebox(0,0)[lb]{\smash{{\SetFigFont{12}{14.4}{\rmdefault}{\mddefault}{\updefault}{\color[rgb]{0,0,0}$t=t_1$}%
}}}}
\end{picture}
}}
\caption{\label{fig2so} Ancestral line corresponding to Fig. \ref{fig1so}. Solide lines represent the descendants of the common ancestor at time $t_i$ in the time period $[t_i,t_{i+1}]$. The bold line represents the ancestral line between $t_0$ and $t_2$. In particular, under the initial configuration of figure \ref{fig1so}, we have $I_{t_0}^6=0$, $I_{t_1}^6=1$,  $I_{t_2}^6=1$.}
\end{figure}
\fi
In this section we recall well-known facts about the common ancestor type distribution in the finite case, and we derive new results reinforcing them. Next, we analyse the asymptotic behaviour of this distribution in the deterministic limit setting. With this in mind, we define $h_k^N$ as the probability that the common ancestor at time $0$ is of type $0$, given that the initial population contains exactly $k$ fit individuals, i.e.
$h_k^N:=P(I_0^N=0\mid X_0^N=k).$
Equivalently, $h_k^N$ is the probability that an initial population of $k$ fit individuals is ancestral to the whole population at some later time.
\begin{remark}
For all $t\geq 0$, the processes $\Theta^{t,N}$ and $\Theta^{0,N}$, under the same initial configuration have the same law. In particular,  
$P(I_t^N=0\mid X_t^N=k)=h_k^N$. 
\end{remark}
\subsection{\texorpdfstring{The case without mutation: $s>0$ and $u=0$}{}}\label{wma}
In this case all the involved quantities can be explicitly computed. Indeed, since there is no mutation, $h_k^N$ is the fixation probability of type $0$ given that in the initial population there are $k$ fit individuals. Thus, setting $T_k^N:=\inf\{t\geq 0: X_t^N=k\}$, \cite[Theorem 6.1]{Du08} yields
\begin{equation*}
 h_k^N=P\left(T_N^N<T_0^N\mid Z_0^N=\frac{k}{N}\right)=\frac{(1+s)^{N-k}-(1+s)^N}{1-(1+s)^N},\quad k\in [N]_0.
\end{equation*}
In other words, if $\Gs^N$ denotes a geometric random variable with parameter $s/(1+s)$ conditioned to be smaller or equal than $N$, then
\begin{equation}\label{rhwm}
 h_k^N=P({\Gs}^N\leq k).
\end{equation}
Since, $\Gs^N\xrightarrow[N\rightarrow\infty]{(d)}\Gs\sim\textrm{Geom}(s/(1+s))$, we get
\begin{equation}\label{hkwuk}
h_k^N\xrightarrow[N\rightarrow\infty]{}1-(1+s)^{-k},\quad k\geq 0.
\end{equation}
This means that the probability $h_k^N$ converges to a strictly positive number, although the initial proportion of fit individuals becomes negligible when $N$ tends to infinity. Moreover, we also have
\begin{equation}\label{hkwukn}
N\geq k_N\xrightarrow[N\rightarrow\infty]{}\infty\Rightarrow h_{k_N}^N\xrightarrow[N\rightarrow\infty]{}1.
\end{equation}
These remarkable facts are in contrast with the behaviour of the corresponding probabilities in the neutral case, where $h_k^N=k/N$. 

On the other hand, if $L_N$ is a binomial random variable with parameters $N$ and $s/(1+s)$ conditioned to be strictly positive, the right hand side of \eqref{rhwm} becomes
\begin{align*}
P(\Gs^N\leq k)&=\sum\limits_{\ell=1}^N P(L_N=\ell)\,\frac{\binom{N}{k}-\binom{N-\ell}{k}}{\binom{N}{k}}=\sum\limits_{\ell=1}^N P(L_N=\ell)\,\frac{\sum\limits_{j=N-\ell}^{N-1}\binom{j}{k}}{\binom{N}{k}}\\
&=\sum\limits_{n=0}^{N-k} P(L_N>n)\frac{k}{N} \,\prod\limits_{j=0}^{n-1}\frac{N-k-j}{N-1-j}.
\end{align*}
Therefore, defining $a_n^N:=P(L_N>n)$, we have
\begin{equation*}
h_k^N= \sum\limits_{n=0}^{N-k}a_n^N\frac{k}{N} \,\prod\limits_{j=0}^{n-1}\frac{N-k-j}{N-1-j}\quad\textrm{and}\quad a_n^N\xrightarrow[N\rightarrow\infty]{}1.
\end{equation*}
\begin{remark}
The convergence results \eqref{hkwuk} and \eqref{hkwukn} also differ from the diffusion limit case, where (see \cite{KB13} and \cite{Ki62})
$$\frac{k_N}{N}\xrightarrow[N\rightarrow\infty]{} x\in[0,1]\Rightarrow h_{k_N}^N\xrightarrow[N\rightarrow\infty]{}\frac{1-e^{-2\sigma x}}{1-e^{-2\sigma}}.$$
\end{remark}
\subsection{The case with mutation}
Henceforth, we assume that $s,u>0$. In contrast to the previous case, when we introduce mutations, there is no fixation of types anymore. Thus a different approach has to be taken. The problem was solved in \cite{KHB13} by studying the Markov process $D^N:=(D^{0,N},D^{1,N},X^{N})$, where $D^{i,N}_t$ holds the number of descendants of type $i$ at time $t$ of an unordered sample with initial composition $(D^{0,N}_0,D^{1,N}_0)$. The relation between $D^N$ and $h_k^N$ is given by
\begin{equation*}
 h_k^N=P\left(\exists t\geq 0:\, D_t^{0,N}+D_t^{1,N}=N\,\mid\,D_0^N=(k,0,k)\right),\quad k\in[N]_0.
\end{equation*}
This representation and a first-step analysis were used in \cite{KHB13} in order to obtain \eqref{hkak} and the recursion {\eqref{an}}$_{n=0}^{N-1}$ defined in Section \ref{catkf}.

\subsubsection{\texorpdfstring{A first characterization of the coefficients $a_k^N$}{}} The inconvenience in the recursion {\eqref{an}}$_{n=0}^{N-1}$ is that it depends on the value of $h_{N-1}^N$. Equation \eqref{an} is a consequence of a probabilistic argument. By contrast, \eqref{a0} and \eqref{a1} follow directly from \eqref{hkak}, plugging in $k=N$ and $k=N-1$ respectively, and they are particular cases of the next result.
\begin{lemma}[Inversion formula]\label{invf}
For all $\ell\in[N]$, we have
\begin{equation*}
 a_{N-\ell}^N=\sum\limits_{k=\ell}^N(-1)^{k+\ell}\binom{k-1}{\ell-1}\binom{N}{k}h_k^N.
\end{equation*}
\end{lemma}
\begin{proof}
From \eqref{hkak}, we see that
 \begin{align*}
\sum\limits_{k=\ell}^N (-1)^{k+\ell}\binom{k-1}{\ell-1}\binom{N}{k}h_k^N&=\sum\limits_{k=\ell}^N (-1)^{k+\ell}\binom{k-1}{\ell-1}\sum\limits_{n=0}^{N-k} a_n^{N}\binom{N-n-1}{k-1}\\
&=\sum\limits_{n=0}^{N-\ell} a_n^{N}\sum\limits_{k=\ell}^{N-n} (-1)^{k+\ell}\binom{k-1}{\ell-1}\binom{N-n-1}{k-1}\\
&=\sum\limits_{m=\ell}^{N} a_{N-m}^{N}\sum\limits_{j=\ell-1}^{m-1} (-1)^{j+\ell-1}\binom{j}{\ell-1}\binom{m-1}{j}.
 \end{align*}
Thus, the result follows from the following combinatorial identity (see e.g. \cite{Rio64}):
$$\sum\limits_{j=\ell-1}^{m-1} (-1)^{j+\ell-1}\binom{j}{\ell-1}\binom{m-1}{j}=\delta_{\ell,m},$$
where $\delta_{\ell,m}$ denotes the Kronecker delta. 
\end{proof}
The aim now is to replace Equation \eqref{a1} by another one, independent of \eqref{a0} and \eqref{an}, for $ n\in[N-1]_2$, and not involving the values of $h_k^N$.
\begin{lemma}[The missing equation]\label{lme}
We have
\begin{equation}\label{me}\tag{$E_N$}
 \left(1+u+\frac{s}{N}\right)\,a_{N-1}^N=\frac{s}{N}\,a_{N-2}^N.
\end{equation}
\end{lemma}
\begin{proof}
First, let us denote $\psi_k^N:=h_k^{N}-\frac{k}{N}$. From Lemma \ref{invf} and the following combinatorial identities, which can be derived from the binomial theorem:
$$\sum\limits_{k=1}^N(-1)^{k+1}\binom{N}{k}k=0\quad\textrm{and}\quad \sum\limits_{k=2}^N(-1)^{k}\binom{N}{k}(k-1)k=0,$$
we deduce that
$$a_{N-1}^N=\sum\limits_{k=1}^N(-1)^{k+1}\binom{N}{k}\psi_k^N\quad\textrm{and}\quad a_{N-2}^N=\sum\limits_{k=2}^N(-1)^{k}\binom{N}{k}(k-1)\psi_k^N.$$
The previous identities and the definition of $\lambda_k^N$ and $\mu_k^N$ lead to
\begin{align*}
\Delta_N:&= \left(1+u+\frac{s}{N}\right)\,a_{N-1}^N-\frac{s}{N}\,a_{N-2}^N=\sum\limits_{k=1}^N(-1)^{k+1}\binom{N}{k}\psi_k^N\left(1+u+\frac{s\, k}{N}\right)\\
&=\sum\limits_{k=1}^N(-1)^{k+1}\binom{N}{k}\psi_k^N\left(\frac{\lambda_k^N}{N-k}+\frac{\mu_k^N}{k}\right).
 \end{align*}
Defining $\tilde{\psi}_0^N:=\tilde{\psi}_N^N:=0$, and, for $k\in[N-1]$, $\tilde{\psi}_k^N:=\frac{\psi_{k}}{k(N-k)}$, we have
\begin{align*}
\sum\limits_{k=1}^N(-1)^{k+1}\binom{N}{k}\psi_k^N\frac{\lambda_k^N}{N-k}& =N\sum\limits_{k=1}^{N-1}(-1)^{k+1}\binom{N-1}{k-1}\tilde{\psi}_k^N\lambda_k^N,
\end{align*}
and
\begin{align*}
\sum\limits_{k=1}^N(-1)^{k+1}\binom{N}{k}\psi_k^N\frac{\mu_k^N}{k}& =N\sum\limits_{k=1}^{N-1}(-1)^{k+1}\binom{N-1}{k}\tilde{\psi}_k^N\mu_k^N.
\end{align*}
As a consequence, we obtain
\begin{equation}\label{dn}
\frac{\Delta_N}{N}= \sum\limits_{k=1}^{N-1}(-1)^{k+1}\binom{N-1}{k-1}\tilde{\psi}_k^N\lambda_k^N+\sum\limits_{k=1}^{N-1}(-1)^{k+1}\binom{N-1}{k}\tilde{\psi}_k^N\mu_k^N.
\end{equation}
In addition, we know from \cite[Eq. (25), (26) and (27)]{KHB13} that, for all $k\in[N-1]$:
\begin{equation}\label{rtpsi}
(\lambda_k^N+\mu_k^N)\,\tilde{\psi}_k^N =\lambda_{k+1}^N\,\tilde{\psi}_{k+1}^N +\mu_{k-1}^N\,\tilde{\psi}_{k-1}^N+\frac{s}{N^2} 
\end{equation}
Multiplying \eqref{rtpsi} by $(-1)^{k+1}\binom{N-2}{k-1}$ and performing the sum over $k\in[N-1]$ yields
\begin{align*}
 \sum\limits_{k=1}^{N-1}&(-1)^{k+1}\binom{N-2}{k-1}(\lambda_k^N+\mu_k^N)\,\tilde{\psi}_k^N=\sum\limits_{k=1}^{N-1}(-1)^{k+1}\binom{N-2}{k-1}\lambda_{k+1}^N\,\tilde{\psi}_{k+1}^N\\
 &+\sum\limits_{k=1}^{N-1}(-1)^{k+1}\binom{N-2}{k-1}\mu_{k-1}^N\,\tilde{\psi}_{k-1}^N +\frac{s}{N^2}\sum\limits_{k=1}^{N-1}(-1)^{k+1}\binom{N-2}{k-1}.
\end{align*}
The last sum equals zero as a consequence of the binomial theorem. Rearranging the sums, we obtain
$$\sum\limits_{k=1}^{N-1}(-1)^{k+1}\binom{N-1}{k-1}\lambda_k^N\tilde{\psi}_k^N+\sum\limits_{k=1}^{N-1}(-1)^{k+1}\binom{N-1}{k}\mu_k^N\tilde{\psi}_k^N=0.$$
This identity together with \eqref{dn} implies that $\Delta_N=0$ and the proof is completed.
\end{proof}
The next result tells us that the coefficients $(a_n^N)_{n=0}^{N-1}$ are characterised by the equations \eqref{a0} and \eqref{an}$_{n\in[N]_2}$.
\begin{lemma}[Uniqueness and positivity of the coefficients]\label{uap}
The system of equations given by \eqref{a0} and \eqref{an}$_{n\in[N]_2}$ has a unique solution $(a_n^N)_{n=0}^{N-1}$, which is in addition, coordinate by coordinate, strictly positive.
\end{lemma}
\begin{proof} We first write the underlying system of equations as follows:
\begin{equation}\label{axb}
\begin{pmatrix}
       d_0&c_0&0&\cdots&0\\
       b_1&d_1&c_1&\ddots&\vdots\\
       0&\ddots&\ddots&\ddots&0\\
       \vdots&\ddots&b_{N-2}&d_{N-2}&c_{N-2}\\
       0&\cdots&0&b_{N-1}&d_{N-1}\\
\end{pmatrix}
\begin{pmatrix}
       a_0^N\\
       a_1^N\\
       \vdots\\
       a_{N-2}^N\\
       a_{N-1}^N\\
\end{pmatrix}=\begin{pmatrix}
       1\\
       0\\
       \vdots\\
       0\\
       0\\
\end{pmatrix},
\end{equation}
where $d_0:=1$, $c_0:=0$,
$$d_n:=\frac{n+1}{N}+\frac{(N-n)s}{N}+u,\quad b_n:=-\frac{(N-n)s}{N},\quad n\in[N-1],$$
and
$$c_n:=-\frac{n+1}{N}-u\nu_1,\quad n\in[N-1].$$
The matrix $A_N$ in \eqref{axb} is strictly diagonally dominant, and hence invertible by the L\'{e}vy-Desplanques Theorem (see for ex. \cite[Theorem 6.1.10]{HJ85}). Thus, \eqref{axb} has a unique solution given by the first column of $A_N^{-1}$. It remains to prove its positivity.

Since $A_N$ is strictly diagonally dominant and all its diagonal entries are strictly positive, the Gerschgorin circle Theorem  implies that all the eigenvalues of $A_N$ have strictly positive real parts (see \cite[Theorem 6.1.1]{HJ85}). In addition, $A_N$ is tridiagonal and $b_i\,c_{i-1}\geq0,$ for all $i\in[N-1]$. Therefore, all its eigenvalues are real (see \cite[p. 174, Problem 5]{HJ85}). Summarizing, all the eigenvalues of $A_N$ are real and strictly positive. The same holds for the sub-matrices $A_N^{(n)}$ consisting of the first $n$ rows and columns of $A_N$ (with $A_N^{(n)}=A_N$). Then, for each $n\in[N]$, $\theta_n^N:=\det(A_N^{(n)})>0$.

Furthermore, since $A_N$ is tridiagonal, \cite[Theorem 2]{Us94} yields
$$(A_N^{-1})_{1,1}=\frac{\phi_2^N}{\theta_N^N}\quad\textrm{and}\quad(A_N^{-1})_{i,1}=|b_1\times\cdots \times b_i|\frac{\phi_{i+1}^N}{\theta_N^N},$$
where $\phi_{N+1}^N=1$, $\phi_{N+2}^N=0$, and from \cite[Lemma 2]{Us94}
$$\theta_k^N\phi_{k+1}^N=\theta_N^N+b_{k} c_{k-1}\theta_{k-1}^N\phi_{k+2}^N,\quad k\in[N].$$
This recursion also tells us that the coefficients $\phi_k^N$ are all positive, and therefore the same holds for the first column of $A_N^{-1}$. The proof is completed.
\end{proof}
\begin{proposition}\label{chan}
There is a random variable $L_N$ with values on $[N]$ such that
$$a_n^N=P(L_N> n),\qquad n\in[N-1]_0.$$
\end{proposition}
\begin{proof}
We claim that the function $n\mapsto a_n^N$ is decreasing. If this is true, we define $\rho_n^N:=a_{n-1}^N-a_n^{N}$, $n\in[N-1]$ and $\rho_{N}^N:=a_{N-1}^N$,
and deduce, for all $n\in[N]$, that
$$\rho_n^N\geq 0\quad\textrm{ and }\quad \sum\limits_{k=1}^N\rho_k^N=a_0^N=1.$$
Consequently, there is a random variable $L_N$ with values in $[N]$ satisfying that $P(L_N=k)=\rho_k^N$. The desired result follows. It remains to prove our claim.\\
From \eqref{me} and Lemma \ref{uap}, we have $a_{N-2}^N-a_{N-1}^N\geq 0.$
In addition, for $n\in[N-1]_2$,
$$\left(\frac{n}{N}+u\nu_1\right)(a_n^N-a_{n-1}^N)=u\nu_0 a_{n-1}^N+\frac{N-(n-1)}{N}s(a_{n-1}^N-a_{n-2}^N),$$
and the claim follows using a backward induction.
\end{proof}
Let now $(\xi_1,...,\xi_N)$ be a vector of random variables with values in $\{0,1\}^N$ with the following distribution
$$P\left((\xi_1,...,\xi_N)=(i_1,...,i_N)\right):=\frac{P\left(L_N=N-r\right)}{{\binom{N}{r}}},\quad\textrm{if}\quad \sum\limits_{j=1}^N i_j=r,$$
and define $\mathcal{G}^N:=\min\{i\geq 1: \xi_i=0\}.$
\begin{corollary}
 For all $k\in[N]_0$, we have $h_k^N=P(\Gs^N\leq k).$
\end{corollary}
\begin{proof}
The proof is very similar to the case without mutation. Indeed, we can decompose the distribution of $\Gs^N$ in the following way:
 \begin{align*}
P(\Gs^N\leq k)&=\sum\limits_{\ell=1}^N P\left(\sum\limits_{i=1}^N\xi_i=N-\ell\right)P\left(\Gs^N\leq k\,\Big|\,\sum\limits_{i=1}^N\xi_i=N-\ell\right)\\
&=\sum\limits_{\ell=1}^N P(L_N=\ell)\,\frac{\binom{N}{k}-\binom{N-\ell}{k}}{\binom{N}{k}}
=\sum\limits_{n=0}^{N-k} \frac{\binom{N-n-1}{k-1}}{\binom{N}{k}}P(L_N>n).
\end{align*}
The result follows from Proposition \ref{chan} and Equation \eqref{hkak}.
\end{proof}
\subsubsection{\texorpdfstring{Asymptotic behaviour of the probabilities $h_k^N$}{}} 
In order to understand the limit behaviour of the common ancestor type distribution, we first study the coefficients $(a_n^N)_{n=0}^{N-1}$. Let us assume for a moment that these coefficients admit a limit when $N$ converges to infinity. In this case, if we fix $n\geq 2$ and we take the limit when $N$ tends to infinity in \eqref{an}, we see that the limit coefficients $(a_k)_{k\geq0}$ should satisfy the following recurrence relation 
\begin{equation}\label{anl}
0=u\nu_1\,a_n-\left(s+u\right)\,a_{n-1}+s\,a_{n-2}.
\end{equation}
\begin{lemma}\label{ane}
The solution of \eqref{anl} has the form 
$$a_n= \frac{(a_0\ell_- -a_1)\,\ell_+^n}{\ell_- -\ell_+} + \frac{(a_1-\ell_+ a_0)\,\ell_-^n}{\ell_- -\ell_+},$$
where
$\ell_\pm=\frac{s+u\pm\sqrt{(s-u)^2+4su\nu_0}}{2u\nu_1}.$
\end{lemma}
\begin{proof}
Note that \eqref{anl} is a homogeneous linear recurrence relation of second order with constant coefficients. Thus, its solution has the form
$a_n=c_1 \,\ell_+^n+c_2\,\ell_-^n$, where $\ell_+$ and $\ell_-$ are the roots of the polynomial $p(x)=u\nu_1 x^2-(s+u)x+s$ and $c_1,\, c_2$ are constants. The previous equation for $n=0$ and $n=1$ permits to determine the values of the constants $c_1,\, c_2$ in terms of $a_0$ and $a_1$. The result follows.
\end{proof}
\begin{remark}\label{ellx}
Note that $\ell_+=1/(1-x_0^+)\in(1,\infty)$ and $\ell_-=s(1-x_0^+)/u\nu_1\in(0,1)$, where $x_0^+$ is defined in \eqref{xo+-}.
\end{remark}
\begin{proposition}[Convergence of the coefficients]\label{cc}
For all $k\in[N]_0$, we have
$$a_k^N\xrightarrow[N\rightarrow\infty]{} a_k=\ell_-^k.$$
\end{proposition}
\begin{proof}
Since $a_0^N=1$ for all $N\geq 1$, the result is true for $k=0$. If we prove the result for $k=1$, then using 
the recurrence relation \eqref{an}, we can deduce that, for each $k\geq 0$, $a_k^N\xrightarrow[N\rightarrow\infty]{}a_k$, where $(a_k)_{k\geq 0}$ is the solution of \eqref{anl} with $a_0:=1$ and $a_1:=\ell_-$. From Lemma \ref{ane} this solution is given by $a_k=\ell_-^k$ and the result follows.
Therefore, it remains only to prove that $a_1^N\xrightarrow[N\rightarrow\infty]{} \ell_-$. From \eqref{a1} and denoting $\psi_k^N:=h_k^N-\frac{k}{N}$, we have to show that $N\psi_{N-1}^N\xrightarrow[N\rightarrow\infty]{} \ell_-$. 

From \cite[Eq. 30]{KHB13}, \eqref{sd}, and using that $\pi_{Z^N}(k/N)=\pi_{X^N}(k)$, we obtain
\begin{align*}
 N\psi_{N-1}^N&=\frac{s}{\frac{N-1}{N}+(N-1)u\nu_1}\,\left(1-\frac{1}{N}\right)\,\frac{\sum\limits_{n=1}^N (N-n)n(\frac{N-n}{N}+u\nu_1)\pi_{Z^N}(\frac{n}{N})}{\sum\limits_{n=1}^N n(\frac{N-n}{N}+u\nu_1)\pi_{Z^N}(\frac{n}{N})}\\
 &=\frac{s}{\frac{N-1}{N^2}+\frac{N-1}{N}u\nu_1}\,\left(1-\frac{1}{N}\right)\,\frac{\sum\limits_{n=1}^N (1-\frac{n}{N})\frac{n}{N}(1-\frac{n}{N}+u\nu_1)\pi_{Z^N}(\frac{n}{N})}{\sum\limits_{n=1}^N \frac{n}{N}(1-\frac{n}{N}+u\nu_1)\pi_{Z^N}(\frac{n}{N})}\\
 &=\frac{s}{\frac{N-1}{N^2}+\frac{N-1}{N}u\nu_1}\,\left(1-\frac{1}{N}\right)\, \frac{E_{\pi_{Z^N}}\left[Z_1^N(1-Z_1^N)(1-Z_1^N+u\nu_1)\right]}{E_{\pi_{Z^N}}\left[Z_1^N(1-Z_1^N+u\nu_1)\right]}.
 \end{align*}
 Thus, \ref{csd} yields $N\psi_{N-1}^N\xrightarrow[N\rightarrow\infty]{}s(1-x_0^+)/u\nu_1=\ell_-$, ending the proof.
\end{proof}
\begin{remark}
The previous result and Remark \ref{ellx} yield $a_1=s(1-x_0^+)/u\nu_1.$
This expression is similar to its diffusion limit analogue, where $\alpha_1=\sigma(1-\tilde{x})/(1+\theta\nu_1),$
with $\theta$ and $\sigma$ as in Eq. \eqref{dlas}, $\tilde{x}=E_{\pi_Y}[Y^2(1-Y)]/E_{\pi_Y}[Y^2(1-Y)]$, and $\pi_Y$ is the stationary distribution of the Wright-Fisher diffusion (see \cite{Ta07}).
\end{remark}
\begin{corollary}\label{convln}
  We have $L_N\xrightarrow[N\rightarrow\infty]{(d)}L,$
  where $L$ is a geometric random variable with parameter $1-\ell_-$.
 \end{corollary}
 \begin{proof}
  Direct from Proposition \ref{cc}.
 \end{proof}
We have all the ingredients to prove the convergence of the probabilities $h_k^N$.
\begin{theorem}[Convergence of the common ancestor type distribution]\label{chk}
Consider a sequence of integers $(k_N)_{N\geq 1}$ satisfying, for all $N\geq 1$, $k_N\in [N]_0$. Then
$$\lim\limits_{N\rightarrow\infty}\frac{k_N}{N}=x\in(0,1)\Rightarrow \lim\limits_{N\rightarrow\infty} h_{k_N}^N=h(x):=x+x\sum\limits_{n=1}^\infty \ell_-^n (1-x)^n=\frac{x}{1-\ell_-(1-x)}.$$
\end{theorem}
\begin{proof}
Fix $n\in\Nb$. Thanks to Proposition \ref{cc}, if $k_N/N$ converges to $x\in(0,1)$, then
$$a_n^N\,\frac{k_N}{N}\prod\limits_{j=0}^{n-1}\frac{N-k_N-j}{N-1-j}\xrightarrow[N\rightarrow\infty]{}\ell_-^n\,x(1-x)^n.$$
Now, let $N_0$ be such that, for all $N\geq N_0$, $\frac{k_N-1}{N}\geq \frac{x}{2}$. Then, using Proposition \ref{chan}, we get, for all $N\geq N_0$,
$$0\leq a_n^N\,\frac{k_N}{N}\prod\limits_{j=0}^{n-1}\frac{N-k_N-j}{N-1-j}\leq \left(1-\frac{x}{2}\right)^n.$$
The result follows as an application of the dominated convergence theorem.
\end{proof}
\begin{corollary}
We have $\frac{\mathcal{G}_N}{N}\xrightarrow[N\rightarrow\infty]{(d)}\mathcal{G}$, where $\mathcal{G}$ is the random variable with values in $[0,1]$ and density function given by
 $$f_\mathcal{G}(x):=\frac{1-\ell_-}{(1-\ell_-(1-x))^2}.$$
\end{corollary}
\begin{proof}
It is straightforward to show that $h(x)=P(\mathcal{G}\leq x)$. The result follows from Theorem \ref{chk}.
\end{proof}
\section{The common ancestor type distribution: the lookdown ASG approach in the finite case}\label{s5}
In this section, we extend to the finite population framework the construction of the (pruned) lookdown ancestral selection graph (LD-ASG) given in \cite{LKBW15}. Based on this construction, we provide a graphical interpretation to the equation \eqref{hkak} and the recurrence relation \eqref{an}. We also give a graphical meaning to the random variable $L_N$ appearing in Proposition \ref{chan}. We assume in the sequel that the parameter of selection  is strictly positive, i.e. $s>0$.
\subsection{The ancestral selection graph}\label{ASG}
The concept of \textit{ancestral selection graph} (ASG) was introduced in \cite{KroNe97} and \cite{NeKro97}  with the purpose of constructing samples from a present population, together with their ancestries, in the diffusion limit of the Moran model with selection and mutation. We recall here this notion in the finite case and we discuss its relation to the common ancestor type distribution. 

Let us start with a given realisation of the untyped $2$-type Moran model of size $N$ in $[0,\tau]$, i.e with a realisation of the reproduction-mutation process $(\Lambda^N_t)_{t\in[0,\tau]}$. In what follows, we use the letter $t$ for the forward time and $\beta:=\tau-t$ for the backward time. The ASG can be read off as follows (see Fig. \ref{fig3}). We start with a sample $M\subset[N]$ of the population at time $\beta=0$ and we trace back the lines of the potential ancestors. When a neutral arrow joins two individuals in the current set of potential ancestors, a \textit{coalescence event} take place, i.e. the two lines merge into a single one, the one at the tail of the arrow. When a neutral arrow hits from outside a potential ancestor, a \textit{relocation event} occurs, i.e. the hit individual moves to the level at the tail of the arrow. When a selective arrow hits the current set of potential ancestors, the individual that is hit has two possible parents, the \textit{incoming branch} at the tail and the \textit{continuing branch} at the tip. The true parent depends on the type of the incoming branch (see Fig. \ref{fig4}), but for the moment we work without types. These unresolved reproduction events can be of two types: a \textit{branching event} if the selective arrow emanates from an individual outside the current set of potential ancestors, and a \textit{collision event} if the selective arrow links two current potential ancestors. The number of potential ancestors decreases by one in a coalescence event, increases by one in a branching event, and remains unchanged in collision and relocation events. The previous procedure provides, at any time $\beta\in[0,\tau]$, the corresponding set of potential ancestors of the initial sample $M$, which we denote by $\As_{[0,\tau]}^{M,N}(*,\beta)$, where $*$ stands for untyped and will be replaced later by the initial (at $t=0$) configuration of types ($\As_{[0,\tau]}^{M,N}(*,0)=M$). 
\ifpdf
\begin{figure}[!ht]
\begin{picture}(0,0)%
\centerline{\resizebox*{10cm}{3cm}{\includegraphics{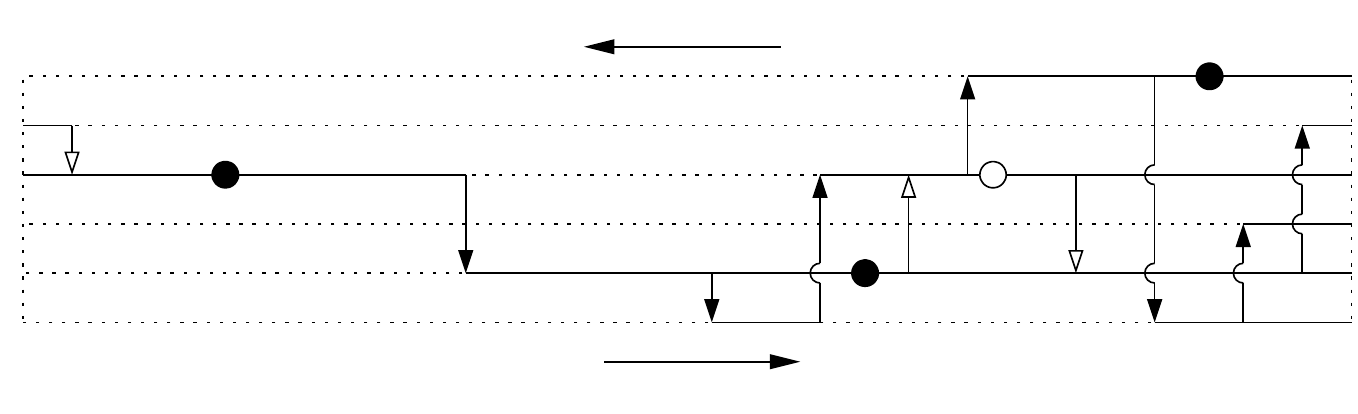}}}
\end{picture}%
\setlength{\unitlength}{4144sp}%
\begingroup\makeatletter\ifx\SetFigFont\undefined%
\gdef\SetFigFont#1#2#3#4#5{%
  \reset@font\fontsize{#1}{#2pt}%
  \fontfamily{#3}\fontseries{#4}\fontshape{#5}%
  \selectfont}%
\fi\endgroup%
\centerline{\resizebox*{10cm}{3cm}{
\begin{picture}(6192,1888)(121,-926)
\put(6256,-691){\makebox(0,0)[lb]{\smash{{\SetFigFont{12}{14.4}{\rmdefault}{\mddefault}{\updefault}{\color[rgb]{0,0,0}$t=\tau$}%
}}}}
\put(136,749){\makebox(0,0)[lb]{\smash{{\SetFigFont{12}{14.4}{\rmdefault}{\mddefault}{\updefault}{\color[rgb]{0,0,0}$\beta=\tau$}%
}}}}
\put(6256,749){\makebox(0,0)[lb]{\smash{{\SetFigFont{12}{14.4}{\rmdefault}{\mddefault}{\updefault}{\color[rgb]{0,0,0}$\beta=0$}%
}}}}
\put(3196,-871){\makebox(0,0)[lb]{\smash{{\SetFigFont{12}{14.4}{\rmdefault}{\mddefault}{\updefault}{\color[rgb]{0,0,0}$t$}%
}}}}
\put(3061,839){\makebox(0,0)[lb]{\smash{{\SetFigFont{12}{14.4}{\rmdefault}{\mddefault}{\updefault}{\color[rgb]{0,0,0}$\beta$}%
}}}}
\put(136,-691){\makebox(0,0)[lb]{\smash{{\SetFigFont{12}{14.4}{\rmdefault}{\mddefault}{\updefault}{\color[rgb]{0,0,0}$t=0$}%
}}}}
\end{picture}
}}
\caption{\label{fig3} Untyped ASG corresponding to Fig. \ref{fig1so} starting with the total population at backward time $\beta=0$.}
\end{figure}
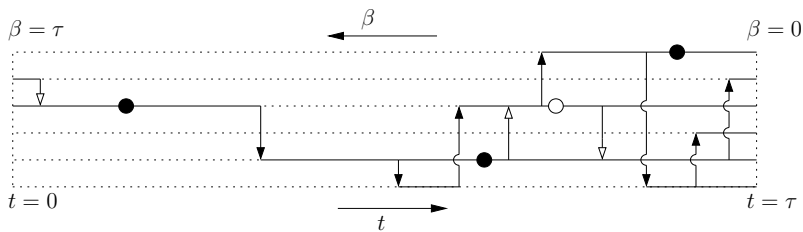
\fi
The untyped ASG in $[0,\tau]$ of the sample $M$, $\As_{[0,\tau]}^{M,N}(*)$, consists of
\begin{enumerate}
 \item the set $V_\tau^{M,N}:=\bigcup_{\beta\in[0,\tau]}\{\beta\}\times\As_{[0,\tau]}^{M,N}(*,\beta)\subset[0,\tau]\times[N]$.
 \item the configuration of arrows  and circles involving the lines in $V_\tau^{M,N}$.
\end{enumerate}
When $M=[N]$, i.e. when we sample the ancestry of the whole population, we simply write  $\As_{[0,\tau]}^{N}(*)$  instead of $\As_{[0,\tau]}^{[N],N}(*)$, and $\As_{[0,\tau]}^{N}(*,v)$ instead of $\As_{[0,\tau]}^{[N],N}(*,v)$.

The true ancestry of the initial sample can be derived after assigning types, $J\in\{0,1\}^N$, to the individuals at time $\beta=\tau$ using the following rule: propagate types forward in time in the ASG and keep track of the changes by respecting the mutation events. At every selective arrow, the incoming branch is the ancestor if it is of type $0$, otherwise the ancestor is the continuing branch (see Fig. \ref{fig4}).
\ifpdf
\begin{figure}[!ht]
\begin{picture}(0,0)%
\centerline{\resizebox*{10cm}{1.2cm}{\includegraphics{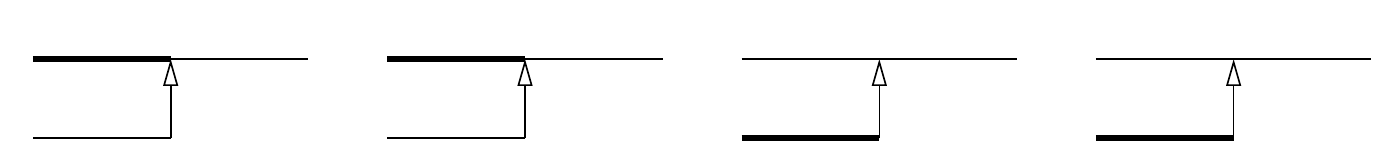}}}
\end{picture}%
\setlength{\unitlength}{4144sp}%
\begingroup\makeatletter\ifx\SetFigFont\undefined%
\gdef\SetFigFont#1#2#3#4#5{%
  \reset@font\fontsize{#1}{#2pt}%
  \fontfamily{#3}\fontseries{#4}\fontshape{#5}%
  \selectfont}%
\fi\endgroup%
\centerline{\resizebox*{10cm}{1.2cm}{
\begin{picture}(6330,718)(211,-161)
\put(2161,434){\makebox(0,0)[lb]{\smash{{\SetFigFont{10}{12.0}{\rmdefault}{\mddefault}{\updefault}{\color[rgb]{0,0,0}$C$}%
}}}}
\put(2161, 29){\makebox(0,0)[lb]{\smash{{\SetFigFont{10}{12.0}{\rmdefault}{\mddefault}{\updefault}{\color[rgb]{0,0,0}$I$}%
}}}}
\put(2926,434){\makebox(0,0)[lb]{\smash{{\SetFigFont{10}{12.0}{\rmdefault}{\mddefault}{\updefault}{\color[rgb]{0,0,0}$D$}%
}}}}
\put(1846,209){\makebox(0,0)[lb]{\smash{{\SetFigFont{12}{14.4}{\rmdefault}{\mddefault}{\updefault}{\color[rgb]{0,0,0}$0$}%
}}}}
\put(3781,434){\makebox(0,0)[lb]{\smash{{\SetFigFont{10}{12.0}{\rmdefault}{\mddefault}{\updefault}{\color[rgb]{0,0,0}$C$}%
}}}}
\put(3781, 29){\makebox(0,0)[lb]{\smash{{\SetFigFont{10}{12.0}{\rmdefault}{\mddefault}{\updefault}{\color[rgb]{0,0,0}$I$}%
}}}}
\put(4546,434){\makebox(0,0)[lb]{\smash{{\SetFigFont{10}{12.0}{\rmdefault}{\mddefault}{\updefault}{\color[rgb]{0,0,0}$D$}%
}}}}
\put(3466,209){\makebox(0,0)[lb]{\smash{{\SetFigFont{12}{14.4}{\rmdefault}{\mddefault}{\updefault}{\color[rgb]{0,0,0}$1$}%
}}}}
\put(3466,-106){\makebox(0,0)[lb]{\smash{{\SetFigFont{12}{14.4}{\rmdefault}{\mddefault}{\updefault}{\color[rgb]{0,0,0}$0$}%
}}}}
\put(5401,434){\makebox(0,0)[lb]{\smash{{\SetFigFont{10}{12.0}{\rmdefault}{\mddefault}{\updefault}{\color[rgb]{0,0,0}$C$}%
}}}}
\put(5401, 29){\makebox(0,0)[lb]{\smash{{\SetFigFont{10}{12.0}{\rmdefault}{\mddefault}{\updefault}{\color[rgb]{0,0,0}$I$}%
}}}}
\put(6166,434){\makebox(0,0)[lb]{\smash{{\SetFigFont{10}{12.0}{\rmdefault}{\mddefault}{\updefault}{\color[rgb]{0,0,0}$D$}%
}}}}
\put(5086,209){\makebox(0,0)[lb]{\smash{{\SetFigFont{12}{14.4}{\rmdefault}{\mddefault}{\updefault}{\color[rgb]{0,0,0}$1$}%
}}}}
\put(5086,-106){\makebox(0,0)[lb]{\smash{{\SetFigFont{12}{14.4}{\rmdefault}{\mddefault}{\updefault}{\color[rgb]{0,0,0}$0$}%
}}}}
\put(541,434){\makebox(0,0)[lb]{\smash{{\SetFigFont{10}{12.0}{\rmdefault}{\mddefault}{\updefault}{\color[rgb]{0,0,0}$C$}%
}}}}
\put(541, 29){\makebox(0,0)[lb]{\smash{{\SetFigFont{10}{12.0}{\rmdefault}{\mddefault}{\updefault}{\color[rgb]{0,0,0}$I$}%
}}}}
\put(1306,434){\makebox(0,0)[lb]{\smash{{\SetFigFont{10}{12.0}{\rmdefault}{\mddefault}{\updefault}{\color[rgb]{0,0,0}$D$}%
}}}}
\put(226,209){\makebox(0,0)[lb]{\smash{{\SetFigFont{12}{14.4}{\rmdefault}{\mddefault}{\updefault}{\color[rgb]{0,0,0}$1$}%
}}}}
\put(226,-106){\makebox(0,0)[lb]{\smash{{\SetFigFont{12}{14.4}{\rmdefault}{\mddefault}{\updefault}{\color[rgb]{0,0,0}$1$}%
}}}}
\put(1846,-106){\makebox(0,0)[lb]{\smash{{\SetFigFont{12}{14.4}{\rmdefault}{\mddefault}{\updefault}{\color[rgb]{0,0,0}$1$}%
}}}}
\put(3286,209){\makebox(0,0)[lb]{\smash{{\SetFigFont{12}{14.4}{\rmdefault}{\mddefault}{\updefault}{\color[rgb]{0,0,0}$0$}%
}}}}
\put(4906,209){\makebox(0,0)[lb]{\smash{{\SetFigFont{12}{14.4}{\rmdefault}{\mddefault}{\updefault}{\color[rgb]{0,0,0}$0$}%
}}}}
\put(6526,209){\makebox(0,0)[lb]{\smash{{\SetFigFont{12}{14.4}{\rmdefault}{\mddefault}{\updefault}{\color[rgb]{0,0,0}$0$}%
}}}}
\put(1666,209){\makebox(0,0)[lb]{\smash{{\SetFigFont{12}{14.4}{\rmdefault}{\mddefault}{\updefault}{\color[rgb]{0,0,0}$1$}%
}}}}
\end{picture}
}}
\caption{\label{fig4} Incoming branch (I), continuing branch (C), and descendant (D). The bold line represents the ancestor.}
\end{figure}
\fi
The resulting object is the ASG with types and is denoted by  $\As_{[0,\tau]}^{M,N}(J)$ (see Fig. \ref{fig5}). For $v\in[0,\tau]$, $\As_{[0,\tau]}^{M,N}(J,v)$ represents the set of ancestors at time $\beta=v$ of the sample $M$, given the configuration of types $J$ at time $\beta=\tau$. As before, and in what follows, the upper index $(M,N)$ will be replaced by $N$ when $M=[N]$. 
\ifpdf
\begin{figure}[!ht]
\begin{picture}(0,0)%
\centerline{\resizebox*{10cm}{3cm}{\includegraphics{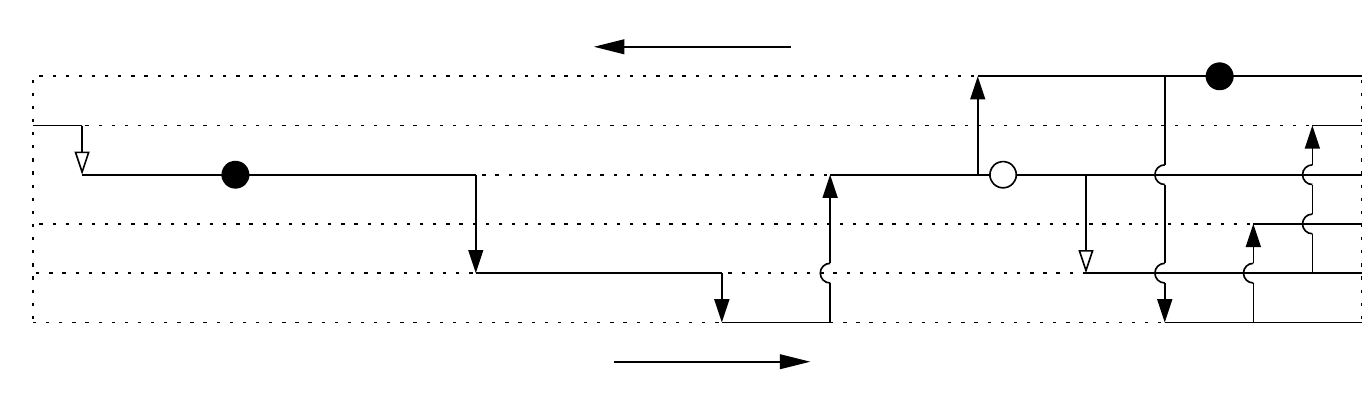}}}
\end{picture}%
\setlength{\unitlength}{4144sp}%
\begingroup\makeatletter\ifx\SetFigFont\undefined%
\gdef\SetFigFont#1#2#3#4#5{%
  \reset@font\fontsize{#1}{#2pt}%
  \fontfamily{#3}\fontseries{#4}\fontshape{#5}%
  \selectfont}%
\fi\endgroup%
\centerline{\resizebox*{10.1cm}{3cm}{
\begin{picture}(6237,1888)(76,-926)
\put(6256,-691){\makebox(0,0)[lb]{\smash{{\SetFigFont{12}{14.4}{\rmdefault}{\mddefault}{\updefault}{\color[rgb]{0,0,0}$t=\tau$}%
}}}}
\put(136,794){\makebox(0,0)[lb]{\smash{{\SetFigFont{12}{14.4}{\rmdefault}{\mddefault}{\updefault}{\color[rgb]{0,0,0}$\beta=\tau$}%
}}}}
\put(6256,749){\makebox(0,0)[lb]{\smash{{\SetFigFont{12}{14.4}{\rmdefault}{\mddefault}{\updefault}{\color[rgb]{0,0,0}$\beta=0$}%
}}}}
\put(3196,-871){\makebox(0,0)[lb]{\smash{{\SetFigFont{12}{14.4}{\rmdefault}{\mddefault}{\updefault}{\color[rgb]{0,0,0}$t$}%
}}}}
\put(3061,839){\makebox(0,0)[lb]{\smash{{\SetFigFont{12}{14.4}{\rmdefault}{\mddefault}{\updefault}{\color[rgb]{0,0,0}$\beta$}%
}}}}
\put(136,-781){\makebox(0,0)[lb]{\smash{{\SetFigFont{12}{14.4}{\rmdefault}{\mddefault}{\updefault}{\color[rgb]{0,0,0}$t=0$}%
}}}}
\put( 91,569){\makebox(0,0)[lb]{\smash{{\SetFigFont{12}{14.4}{\rmdefault}{\mddefault}{\updefault}{\color[rgb]{0,0,0}$1$}%
}}}}
\put( 91,344){\makebox(0,0)[lb]{\smash{{\SetFigFont{12}{14.4}{\rmdefault}{\mddefault}{\updefault}{\color[rgb]{0,0,0}$0$}%
}}}}
\put( 91,119){\makebox(0,0)[lb]{\smash{{\SetFigFont{12}{14.4}{\rmdefault}{\mddefault}{\updefault}{\color[rgb]{0,0,0}$1$}%
}}}}
\put( 91,-106){\makebox(0,0)[lb]{\smash{{\SetFigFont{12}{14.4}{\rmdefault}{\mddefault}{\updefault}{\color[rgb]{0,0,0}$1$}%
}}}}
\put( 91,-331){\makebox(0,0)[lb]{\smash{{\SetFigFont{12}{14.4}{\rmdefault}{\mddefault}{\updefault}{\color[rgb]{0,0,0}$0$}%
}}}}
\put( 91,-556){\makebox(0,0)[lb]{\smash{{\SetFigFont{12}{14.4}{\rmdefault}{\mddefault}{\updefault}{\color[rgb]{0,0,0}$0$}%
}}}}
\end{picture}%
}}
\caption{\label{fig5} ASG with types obtained from to Fig. \ref{fig3} after assignment of types at time $\beta=\tau$.}
\end{figure}
\fi
\begin{remark}
In the finite Moran model, collisions occur with strictly positive rate. By contrast, in the diffusion limit setting there are no collision events (see \cite{KroNe97}).
\end{remark}
We denote by $J^N$ the random variable with values on $\{0,1\}^N$ providing the initial configuration of types. Let $P_k$ be the joint law of $\Lambda^N$ and $J^N$, such that under $P_k$, $J^N$ is independent of $\Lambda^N$ and uniform on all possible configurations with exactly $k$ zeros. The relation between $h_k^N$ and the ASG with types is given in the following lemma. 
\begin{lemma}\label{lasg1}
 For all $k\in[N]_0$, we have
 $$h_k^N=\lim\limits_{\tau\rightarrow\infty} P_k\left(\As_{[0,\tau]}^{N}(J^N,\tau)=\{i\},\textrm{ for some $i\in[N]$},\, J_i^N=0\right).$$
\end{lemma}
\begin{proof}
Let $T_a^N:=\inf\{u>0:\Theta^{0,N}(u)\in\As_N\}$, where $\Theta^{0,N}$ is the offspring-type process defined in Section \ref{sapv} with initial configuration given by $\theta_i^{0,N}(0)=\{i\}$ and $j_i^{0,N}(0)=J^N_i$, $i\in[N]$. In other words, $T_a^N$ is the first time that the progeny of one of the individuals at time $t=0$ take over in the population. The mentioned individual is the common ancestor at time $0$. In addition,
$$\left\{\As_{[0,T_a^N]}^{N}(J^N,T_a^N)=\{i\}\right\}=\left\{i \textrm{ is the common ancestor at time $0$},\, j^{0,N}(0)=J^N\right\}.$$
Therefore, we have
$$h_k^N=P_k\left(\As_{[0,T_a^N]}^{N}(J^N,T_a^N)=\{i\},\textrm{ for some $i\in[N]$},\, J_i^N=0\right).$$
Note also that, on $\{\tau>T_a^N\}$, we have
$$\left\{\As_{[0,T_a^N]}^{N}(J^N,T_a^N)=\{i\}\right\}=\left\{\As_{[0,\tau]}^{N}(J^N,\tau)=\{i\}\right\}.$$
Since in addition, $T_a^N$ is almost surely finite, we deduce that
\begin{align*}
h_k^N&=\lim\limits_{\tau\rightarrow\infty}P_k\left(\As_{[0,\tau]}^{N}(J^N,\tau)=\{i\},\textrm{ for some $i\in[N]$},\, J_i^N=0,\, T_a^N<\tau\right)\\
&=\lim\limits_{\tau\rightarrow\infty}P_k\left(\As_{[0,\tau]}^{N}(J^N,\tau)=\{i\},\textrm{ for some $i\in[N]$},\, J_i^N=0\right),
\end{align*}
and the proof is accomplished.
\end{proof}
\subsection{A Markov version of the ASG and the bottlenecks}\label{MASG}
An important fact is that we may construct the untyped ASG in a Markovian way. By this we mean that we can construct a Markov process $\chi^{M,N}:=(\As_\beta^{M,N},\varDelta^{M,N}_\beta)_{\beta\geq 0}$, where\\
$\bullet$ $\As_\beta^{M,N}\subset[N]$ represents the set of potential ancestors at time $\beta$ of an initial sample $M$, i.e the analogue of $\As_{[0,\beta]}^{M,N}(*,\beta)$ given in Section \ref{ASG}.\\
$\bullet$ $\varDelta^{M,N}:=\{\eta_{i}^{0,N}(M,\cdot),\eta_{i}^{1,N}(M,\cdot),\{\eta_{i,j}^{\vartriangle,N}(M,\cdot),\eta_{i,j}^{\blacktriangle,N}(M,\cdot)\}_{j\in[N]/\{i\}} \}_{i\in[N]}$ is a collection of counting processes encoding with their jumps the reproduction and mutation events involving the potential ancestors of $M$.

To see this, we first consider $\Lambda^N:=(\Lambda^N_t)_{t\in\Rb}$ the reproduction-mutation process defined on the entire real line, i.e. the graphical representation of the Moran model between $-\infty$ and $\infty$. Since $\Lambda^N$ is a finite collection of independent Poisson processes, it follows that $$(\Lambda_t^N)_{t\in[-\tau,0]}\overset{(d)}{=}(\Lambda_t^N)_{t\in[0,\tau]},\quad\textrm{and}\quad\hat{\Lambda}^N:=(\Lambda^N_{-t})_{t\in\Rb}\overset{(d)}{=}{\Lambda}^N.$$
These identities in law motivate the following construction of the process $\chi^{M,N}$. We start with the sample $M$ and we read off the configuration of arrows and circles given by $\hat{\Lambda}^N$ as follows: \\
$\bullet$ if $t=-\beta$ is a jump of the process $\lambda_{i,j}^{\vartriangle,N}$ and $j\in\As_{\beta^-}^{M,N}$, we have two options:

$\star$ if $i\in\As_{\beta^-}^{M,N}$, a collision occurs. We set $\As_{\beta}^{M,N}=\As_{\beta^-}^{M,N}$.

$\star$ if $i\notin\As_{\beta^-}^{M,N}$, a branching occurs. We set $\As_{\beta}^{M,N}=\As_{\beta^-}^{M,N}\cup\{i\}$.\\ 
In both cases, we set $\eta_{i,j}^{\vartriangle,N}(M,\beta)=\eta_{i,j}^{\vartriangle,N}(M,\beta)+1$.\\
$\bullet$ If $t=-\beta$ is a jump of $\lambda_{i,j}^{\blacktriangle,N}$ and $j\in\As_{\beta^-}^{M,N}$, we have two possibilities:

 $\star$ if $i\in\As_{\beta^-}^{M,N}$, a coalescence occurs. We set $\As_{\beta}^{M,N}=\As_{\beta^-}^{M,N}/\{j\}$.
 
 $\star$ if $i\notin\As_{\beta^-}^{M,N}$, a relocation occurs and we set $\As_{\beta}^{M,N}=(\As_{\beta^-}^{M,N}/\{j\})/\cup\{i\}$.\\
In both cases, we set $\eta_{i,j}^{\blacktriangle,N}(M,\beta)=\eta_{i,j}^{\blacktriangle,N}(M,\beta^-)+1$.\\
$\bullet$ if $t=-\beta$ is a jump of $\lambda_{i}^{0,N}$ and $i\in\As_{\beta^-}^{M,N}$, a mutation to type $0$ occurs. We set $\As_{\beta}^{M,N}=\As_{\beta^-}^{M,N}$ and $\eta_{i}^{0,N}(M,\beta)=\eta_{i}^{0,N}(M,\beta^-)+1$.\\
$\bullet$ if $t=-\beta$ is a jump of $\lambda_{i}^{1,N}$ and $i\in\As_{\beta^-}^{M,N}$, a mutation to type $1$ happens. We set $\As_{\beta}^{M,N}=\As_{\beta^-}^{M,N}$ and $\eta_{i}^{1,N}(M,\beta)=\eta_{i}^{1,N}(M,\beta^-)+1$.

The so-constructed process $\chi^{M,N}$ is clearly a Markov process, and leads to the following definition of the untyped ASG of the sample $M$ in the interval $[0,\tau]$:
$$\As_{[0,\tau]}^{M,N}(*):=\left(\chi_{\beta}^{M,N}\right)_{\beta\in[0,\tau]}.$$
We call $\chi^{M,N}$ the \textit{ancestral selection process}. Another advantage of this construction is that, if $\tau_1>\tau_2>0$, $\As_{[0,\tau_2]}^{M,N}(*)$ is the restriction of $\As_{[0,\tau_1]}^{M,N}(*)$ to the interval $[0,\tau_2]$. Moreover, we can define the untyped ASG in the entire positive real line as $$\As_{[0,\infty)}^{M,N}(*):=\left(\chi_{\beta}^{M,N}\right)_{\beta\geq 0}.$$

Given a configuration of types, $J\in\{0,1\}^N$, at time $\beta=\tau$, the corresponding ASG with types in $[0,\tau]$, $\As_{[0,\tau]}^{M,N}(J)$, is obtained by propagating types in $\As_{[0,\tau]}^{M,N}(*)$ and extracting the true genealogy as before.
 
Note that the process $K^{M,N}:=(K_\beta^{M,N})_{\beta\geq 0}$, counting the lines in the untyped ASG of the sample $M$, i.e. $K_\beta^{M,N}:=|\As_\beta^{M,N}|$, is a birth-death process with rates:
\begin{equation*}
q_{K^N}(k,k-1):=\frac{k(k-1)}{N},\quad\textrm{and}\quad q_{K^N}(k,k+1):=\frac{k(N-k)s}{N}.
\end{equation*}
As in the diffusion limit setting (see \cite{Popf13}), for $\tau$ sufficiently large, $\As_{[0,\tau]}^{M,N}(*)$ has bottlenecks, i.e. times at which it consists of a single line (see Fig. \ref{fig6}). Indeed, define $T_b^{M,N}:=\inf\{v\geq 0: K_v^{M,N}=1\}$. Since $K^{M,N}$ is an irreducible Markov chain with finite state space, the time $T_b^{M,N}$ is almost surely finite. Moreover, for all $\tau>T_b^{M,N}$, $T_b^{M,N}$ is a bottleneck of $\As_{[0,\tau]}^{M,N}(*)$. 
\ifpdf
\begin{figure}[!ht]
\begin{picture}(0,0)%
\centerline{\resizebox*{7cm}{3.2cm}{\includegraphics{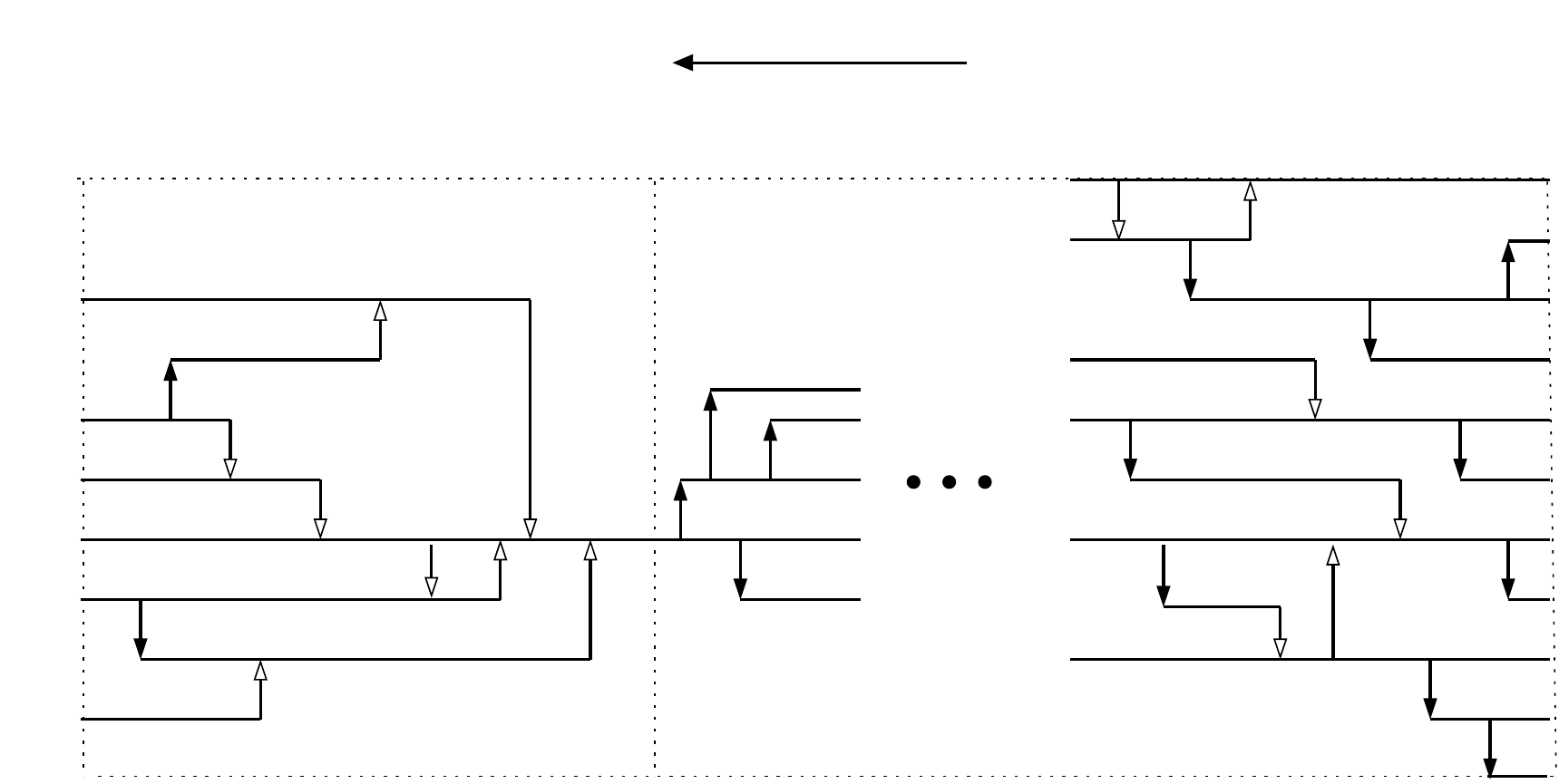}}}%
\end{picture}%
\setlength{\unitlength}{4144sp}%
\begingroup\makeatletter\ifx\SetFigFont\undefined%
\gdef\SetFigFont#1#2#3#4#5{%
  \reset@font\fontsize{#1}{#2pt}%
  \fontfamily{#3}\fontseries{#4}\fontshape{#5}%
  \selectfont}%
\fi\endgroup%
\centerline{\resizebox*{7cm}{3.2cm}{
\begin{picture}(7857,3940)(211,-4283)
\put(226,-1051){\makebox(0,0)[lb]{\smash{{\SetFigFont{12}{14.4}{\rmdefault}{\mddefault}{\updefault}{\color[rgb]{0,0,0}$\beta=\tau$}%
}}}}
\put(3331,-1051){\makebox(0,0)[lb]{\smash{{\SetFigFont{12}{14.4}{\rmdefault}{\mddefault}{\updefault}{\color[rgb]{0,0,0}$\beta=\beta_0$}%
}}}}
\put(7651,-1051){\makebox(0,0)[lb]{\smash{{\SetFigFont{12}{14.4}{\rmdefault}{\mddefault}{\updefault}{\color[rgb]{0,0,0}$\beta=0$}%
}}}}
\put(4096,-466){\makebox(0,0)[lb]{\smash{{\SetFigFont{12}{14.4}{\rmdefault}{\mddefault}{\updefault}{\color[rgb]{0,0,0}$\beta$}%
}}}}
\end{picture}%
}}
\caption{\label{fig6} Untyped ASG with a bottleneck at backward time $\beta=\beta_0$.}
\end{figure}
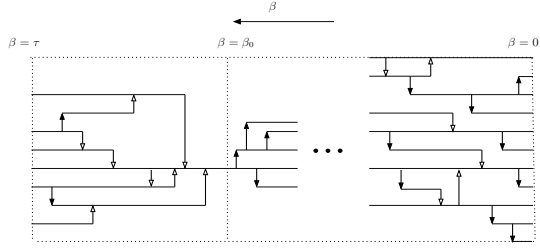
\fi

In particular, if we are interested in the common ancestor type distribution, instead of following the ancestry of the whole population, we can equivalently follow the ancestry of the one individual at the bottleneck $T_b^N$. The following lemma formalises the fact that sooner or later all the lines in the sample coalesce into the ancestral line.
\begin{lemma}\label{lasg2}
For all $k\in[N]_0$, we have
$$h_k^N=\lim\limits_{\tau\rightarrow\infty} P_k\left(\As_{[0,\tau]}^{\{1\},N}(J^N,\tau)=\{i\},\textrm{ for some $i\in[N]$},\, J_i^N=0\right).$$
\end{lemma}
\begin{proof}
 Since $T_b^N$ is almost surely finite, we deduce from Lemma \ref{lasg1} that
 $$h_k^N=\lim\limits_{\tau\rightarrow\infty} P_k\left(\As_{[0,\tau]}^{N}(J^N,\tau)=\{i\},\textrm{ for some $i\in[N]$},\, J_i^N=0,\, \tau>T_b^N\right).$$
 Note that, since $T_b^N$ is a bottleneck, for $\tau>T_b^N$, the true genealogy of $\As_{[0,\tau]}^{N}(*)$ in the interval $[T_b^N,\tau]$ depends only on $(\chi^N_\beta)_{\beta\in[T_b^N,\tau]}$ and on the configuration of types $J^N$ at time $\tau$. In addition, the true genealogy at any time in $[T_b^N,\tau]$ consists of only one individual. In particular, we denote by $i((\chi^N_\beta)_{\beta\in[T_b^N,\tau]},J^N)$ the unique ancestor of the whole population at time $\tau$. Therefore, we have
 \begin{equation}\label{b1}
 h_k^N=\lim\limits_{\tau\rightarrow\infty}P_k\left(J_{i((\chi^N_\beta)_{\beta\in[T_b^N,\tau]},J^N)}^N=0,\, \tau>T_b^N\right). 
 \end{equation}
 Now, we consider $\mathbb{F}^N:=(\Fs^N_\beta)_{\beta\geq 0}$ the natural filtration associated to the process $\hat{\Lambda}^N$, and we note that $T_b^N$ is an $\Fb$-stopping time. Therefore, conditioning on $\Fs_{T_b^N}^N$, and applying the Markov property, we obtain
 \begin{equation}\label{b2}
 P_k\left(J_{i((\chi^N_\beta)_{\beta\in[T_b^N,\tau]},J^N)}^N=0,\, \tau>T_b^N\right)=E_k\left[1_{\{\tau>T_b^N\}}\,p_k\left(\tau-T_b^N,\As_{T_b^N}^N\right)\right],
 \end{equation}
where, for $v>0$ and $j\in[N]$, 
$$p_k(v,\{j\})=P_k\left(J_{i((\chi^{\{j\},N}_\beta)_{\beta\in[0,v]},J^N)}^N=0\right),$$
and $i((\chi^{\{j\},N}_\beta)_{\beta\in[0,v]},J^N)$ is the ancestor at time $v$ of the individual placed at level $j$ at time $0$, given the configuration of types $J^N$. Equivalently, we have
$$\As_{[0,v]}^{\{j\},N}(J^N,v)=\{i((\chi^{\{j\},N}_\beta)_{\beta\in[0,v]},J^N)\}.$$
We conclude that
$$p_k(v,\{j\})=P_k\left(\As_{[0,v]}^{\{j\},N}(J^N,v)=\{i\},\textrm{ for some $i\in[N]$},\, J_i^N=0\right).$$
Moreover, due to the exchangeability of the lines, $p_k(v,\{j\})$ does not depend on $j$, and then $p_k(v,\{j\})=p_k(v,\{1\})$. Plugging this in \eqref{b2} and using \eqref{b1}, we get
 \begin{equation*}
 h_k^N=\lim\limits_{\tau\rightarrow\infty}E_k\left[1_{\{\tau>T_b^N\}}\,p_k\left(\tau-T_b^N,\{1\}\right)\right]. 
 \end{equation*}
 Since $T_b^N$ is almost surely finite, we achieve the proof applying the dominated convergence theorem.
\end{proof}
\begin{remark}
 Due to the exchangeability of the lines, the relocation events do not affect the common ancestor type distribution, and therefore, can be ignored.
\end{remark}
\subsection{Classification of paths in the ASG}\label{cp}
The purpose of this paragraph is to better understand the composition of the untyped ASG and the passage to the ASG with types. Moreover, we would like to discriminate between the relevant and the irrelevant information provided by the ASG. The discussion presented here will serve also as a motivation to introduce new objects encoding more efficiently the common ancestor type distribution. With this purpose, we fix $\tau>0$, $M\subset[N]$ and a realisation of $\As_{[0,\tau]}^{M,N}(*)$. We call a \textit{path} in $\As_{[0,\tau]}^{M,N}(*)$ a subset of $\bigcup_{\beta\in[0,\tau]}\{\beta\}\times\As^{M,N}_\beta$ of the form $$\gamma:=\left[\bigcup\limits_{k=1}^m[\tau_{k-1},\tau_k)\times\{i_k\}\right]\cup\left[\{\tau_m\}\times\{i_{m+1}\}\right],$$ 
where $0=\tau_0<\cdots<\tau_m=\tau$,  $i_k\in[N]$, and for all $i_k\neq i_{k+1}$, $\tau_k$ is a jumping time of $\lambda_{i_{k+1},i_{k}}^{\vartriangle,N}$ or $\lambda_{i_{k+1},i_{k}}^{\blacktriangle,N}$, i.e. there is an arrow going from $i_{k+1}$ to $i_k$.

The path $\gamma$ is said to be \textit{neutral} if it uses only neutral arrows. If the path is not neutral, we denote by $\tau_{k_1}<\dots<\tau_{k_n}$, the times where the selective arrows appear, and $\tau_{k_{n+1}}:=\tau$. We call $\gamma$ \textit{almost neutral} if it is not neutral, there is at least one mutation in $[\tau_{k_n},\tau]$, the first mutation after $\tau_{k_n}$ being to type $0$, and on each interval $[\tau_{k_i},\tau_{k_{i+1}})$ containing mutations, the first mutation after $\tau_{k_i}$ is to type $0$. We say that $\gamma$ is fictitious, if there is an interval $[\tau_{k_i},\tau_{k_{i+1}})$ containing mutations, and such that the first mutation after $\tau_{k_i}$ is to type $1$. Finally, we say that $\gamma$ is \textit{truly selective} if it is nor neutral, nor almost neutral and nor fictitious. Equivalently, $\gamma$  is truly selective if there is no mutation in $[\tau_{k_n},\tau]$, and the restriction of $\gamma$ to $[\tau_{1},\tau_{k_n}]$ either has no mutations or is neutral or almost neutral. Examples of this classification of paths are illustrated in Figure \ref{fig7}.

\ifpdf
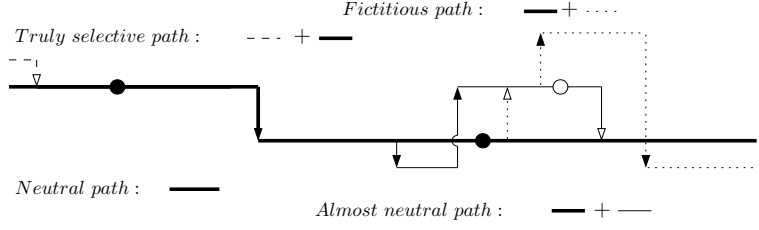
\begin{figure}[!ht]
\begin{picture}(0,0)%
\centerline{\resizebox*{10cm}{3cm}{\includegraphics{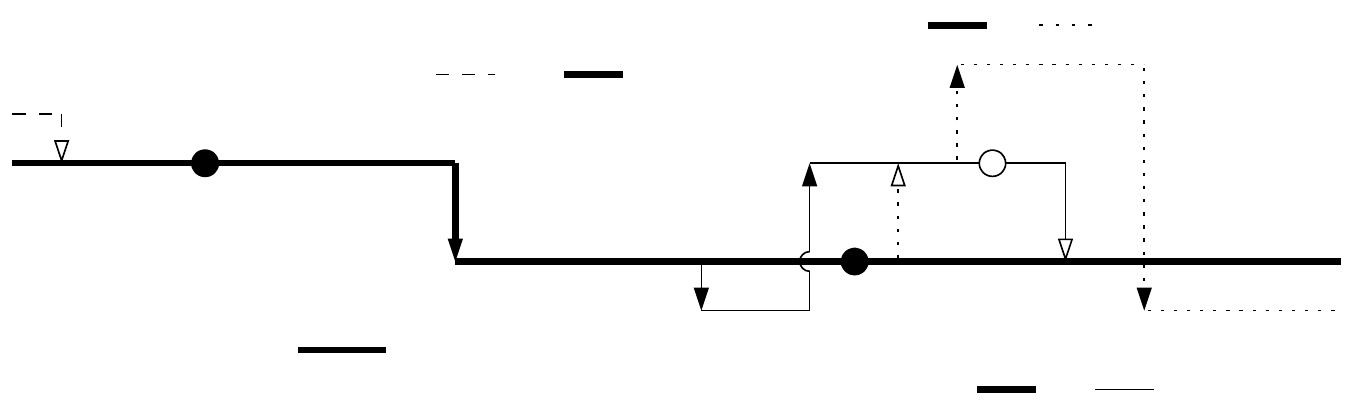}}}
\end{picture}%
\setlength{\unitlength}{4144sp}%
\begingroup\makeatletter\ifx\SetFigFont\undefined%
\gdef\SetFigFont#1#2#3#4#5{%
  \reset@font\fontsize{#1}{#2pt}%
  \fontfamily{#3}\fontseries{#4}\fontshape{#5}%
  \selectfont}%
\fi\endgroup%
\centerline{\resizebox*{10cm}{3cm}{
\begin{picture}(6165,1876)(169,145)
\put(2521,1649){\makebox(0,0)[lb]{\smash{{\SetFigFont{12}{14.4}{\rmdefault}{\mddefault}{\updefault}{\color[rgb]{0,0,0}$+$}%
}}}}
\put(226,389){\makebox(0,0)[lb]{\smash{{\SetFigFont{10}{12.0}{\rmdefault}{\mddefault}{\updefault}{\color[rgb]{0,0,0}$Neutral\,\, path:$}%
}}}}
\put(2926,1874){\makebox(0,0)[lb]{\smash{{\SetFigFont{10}{12.0}{\rmdefault}{\mddefault}{\updefault}{\color[rgb]{0,0,0}$Fictitious\,\, path:$}%
}}}}
\put(4726,1874){\makebox(0,0)[lb]{\smash{{\SetFigFont{12}{14.4}{\rmdefault}{\mddefault}{\updefault}{\color[rgb]{0,0,0}$+$}%
}}}}
\put(226,1649){\makebox(0,0)[lb]{\smash{{\SetFigFont{10}{12.0}{\rmdefault}{\mddefault}{\updefault}{\color[rgb]{0,0,0}$Truly\,\, selective\,\, path:$}%
}}}}
\put(2701,209){\makebox(0,0)[lb]{\smash{{\SetFigFont{10}{12.0}{\rmdefault}{\mddefault}{\updefault}{\color[rgb]{0,0,0}$Almost\,\, neutral\,\, path:$}%
}}}}
\put(4996,209){\makebox(0,0)[lb]{\smash{{\SetFigFont{12}{14.4}{\rmdefault}{\mddefault}{\updefault}{\color[rgb]{0,0,0}$+$}%
}}}}
\end{picture}
}}
\caption{\label{fig7} Examples of neutral, almost neutral, fictitious and truly selective paths extracted from Fig. \ref{fig3}.}
\end{figure}
\fi
Fictitious paths can not be used, independently of the configuration of types at time $\tau$. Hence, fictitious paths are part of the irrelevant information in $\As_{[0,\tau]}^{M,N}(*)$. In order to identify all the irrelevant material, we give a second classification of paths. 

From the construction of the untyped ASG, paths are never hit by a neutral arrow. In an almost neutral path, selective arrows are always used. We call $\gamma$ \textit{irrelevant} if it is fictitious or if it is hit by an almost neutral path. We say that $\gamma$ is relevant if it is not irrelevant. Irrelevant paths are never used, independently of the configuration of types, i.e. they contain irrelevant information on $\As_{[0,\tau]}^{M,N}(*)$. On the contrary, for a relevant path $\gamma$, there is always a configuration of types $J$ such that $\gamma\in\As_{[0,\tau]}^{M,N}(J)$. Motivated by this fact, we define the \textit{relevant untyped ASG} as
\begin{equation}\label{rasg}
\Rs_{[0,\tau]}^{M,N}(*):=\{\gamma\in \As_{[0,\tau]}^{M,N}(*):\,\gamma\textrm{ is a relevant path}\},
\end{equation}
and from the discussion above, we see that
\begin{equation}\label{rasg2}
 \Rs_{[0,\tau]}^{M,N}(*)=\bigcup\limits_{J\in\{0,1\}^N}\As_{[0,\tau]}^{M,N}(J).
\end{equation}
A relevant neutral or relevant almost neutral path, which is hit only by irrelevant or truly selective paths is called \textit{immune}. Immune paths are exactly the true ancestors of the sample $M$, when the configuration of types at time $\tau$ consists of only ones. In particular, there are at most $|M|$ immune paths in $\As_{[0,\tau]}^{M,N}(*)$.
 \subsection{The case without mutation}
We assume in this section that $u=0$. In this situation, as in the diffusion limit case (see \cite{Popf13}), a natural link emerges between the distribution of the common ancestor and the stationary number of lines in the ancestral selection graph. 

Thanks to Lemma \ref{lasg2}, it is sufficient to analyse the ASG starting with a single individual at level $1$. Moreover, since there are no mutations, we have only neutral and truly selective paths in $\As_{[0,\tau]}^{\{1\},N}(*)$. In addition, we have exactly one immune path, which is neutral. All the other paths are truly selective and hit the immune path at some time in $[0,\tau]$. Therefore, the immune path is the common ancestor if and only if all the individuals in $\As_{\tau}^{\{1\},N}$ are of type $1$. Since truly selective paths can be only used by type $0$ individuals, we deduce that the common ancestor is of type $0$ if and only if one of the individuals in $\As_{\tau}^{\{1\},N}$ is of type $0$. Thus, conditioning on the number of lines in $\As_{\tau}^{\{1\},N}$, we get
$$P_k\left(J_i^N=0,\textrm{ where: }\,\As_{[0,\tau]}^{\{1\},N}(J^N,\tau)=\{i\}\right)=\sum\limits_{\ell=1}^N P(K_\tau^{\{1\},N}=\ell)\,\frac{\binom{N}{k}-\binom{N-\ell}{k}}{\binom{N}{k}}.$$
Taking the limit when $\tau$ tends to infinity in the previous expression and using Lemma \ref{lasg2}, we obtain
$$h_k^N=\sum\limits_{\ell=1}^N P(K_\infty^{\{1\},N}=\ell)\,\frac{\binom{N}{k}-\binom{N-\ell}{k}}{\binom{N}{k}}=\frac{k}{N}\sum\limits_{n=0}^{N-k}P(K_\infty^{\{1\},N}>n) \prod\limits_{j=0}^{n-1}\frac{N-k-j}{N-1-j}.$$

Using the detailed balance equation, it follows that $K_\infty^{\{1\},N}$, the stationary number of lines in the untyped ASG, is distributed as a binomial random variable with parameter $N$ and $s/(1+s)$ conditioned to be strictly positive. Thus, we have recovered the results of Section \ref{wma}, and established that  
the random variable $L_N$ corresponds to the stationary number of lines in the ASG.
\subsection{The case with mutation: the relevant ASG}
From now on, we assume that $s,u>0$. In this case, as discussed in Section \ref{cp}, paths which are never used appear in the untyped ASG. The untyped relevant ASG defined in Section \ref{cp} permits to obtain a first graphical interpretation to Eq. \eqref{hkak}. 

We denote by $R_\tau^N$ the number of lines at time $\tau$ in the relevant untyped ASG, $\Rs_{[0,\tau]}^{\{1\},N}$. The graphical representation of $h_k^N$ is given in the next lemma.
\begin{lemma}\label{rasghk}
The random variables $(R_\tau^N)_{\tau>0}$ converge in distribution to a random variable $R_\infty^N$ with values in $[N]$ and
$$h_k^N:=\frac{k}{N}\sum\limits_{n=0}^{N-k}P(R_\infty^{N}>n) \prod\limits_{j=0}^{n-1}\frac{N-k-j}{N-1-j}.$$
\end{lemma}
\begin{proof}
From Lemma \ref{lasg2} and Eq. \eqref{rasg2}, we deduce that
\begin{equation}\label{rasghkf}
 h_k^N=\lim\limits_{\tau\rightarrow\infty} P_k\left(\Rs_{[0,\tau]}^{\{1\},N}(J^N,\tau)=\{i\},\textrm{ for some $i\in[N]$},\, J_i^N=0\right).
\end{equation}
From the definition of $\Rs_{[0,\tau]}^{\{1\},N}(*)$, we see that it consists of one immune path and truly selective paths hitting the immune path at some time between $0$ and $\tau$. If $J^N$ consists of only ones, the ancestral line is the immune path and the type of the true ancestor is $1$. If $J^N$ is not identically one, we have two possibilities: (1) there is only one $0$ in $J^N$ at the immune path, and the ancestral line is the immune path, or (2) there is at least one $0$ in $J^N$ at a truly selective path, then the ancestral line is the truly selective path with type $0$ at time $\tau$, which is not hit by another truly selective path with type $0$ at time $\tau$ (there is always such a path, since the number of truly selective paths is almost surely finite). In both cases, the true ancestor is of type $0$. In conclusion, the true ancestor at time $\tau$ is of type $0$ if and only if one of the lines in the untyped relevant ASG is of type $0$ at time $\tau$. We conclude that
\begin{equation}\label{rasgbl}
 P_k\left(J_i^N=0,\textrm{ where: }\,\Rs_{[0,\tau]}^{\{1\},N}(J^N,\tau)=\{i\}\right)=\sum\limits_{\ell=1}^N P(R_\tau^{N}=\ell)\,\frac{\binom{N}{k}-\binom{N-\ell}{k}}{\binom{N}{k}}.
\end{equation}
Let $(R_{\tau_n}^N)_{n\geq 0}$ be a subsequence of $(R_\tau^N)_{\tau>0}$. Since the involved random variables share the same finite state space, we conclude that $(R_{\tau_n}^N)_{n\geq 0}$ is tight. Therefore, there is a subsequence $(R_{\tau_{n_k}}^N)_{k\geq 0}$ which is convergent in the weak sense. We denote its limit by $R_\infty^N$.  Using this and Equations \eqref{rasghkf} and \eqref{rasgbl}, we obtain
$$h_k^N=\sum\limits_{\ell=1}^N P(R_\infty^{N}=\ell)\,\frac{\binom{N}{k}-\binom{N-\ell}{k}}{\binom{N}{k}}=\frac{k}{N}\sum\limits_{n=0}^{N-k}P(R_\infty^{N}>n) \prod\limits_{j=0}^{n-1}\frac{N-k-j}{N-1-j}.$$
 Lemma \ref{invf} implies that the law of ${R}^{N}$ is uniquely determined by the common ancestor type distribution. Since this holds for any subsequence of $(R_\tau^N)_{\tau>0}$, the result follows.
\end{proof}
From the previous lemma, the random variable $L_N$ given in Proposition \ref{chan} corresponds to the asymptotic number of lines in the relevant untyped ASG. Unfortunately, it is not easy to describe the law of $(R_\infty^N)$. Therefore, we use a different approach in order to give a graphical explanation to the recurrence relation \eqref{an}.
\subsection{The case with mutation: the lookdown ASG}
 The recent work \cite{LKBW15} provides a graphical interpretation to Eq. \eqref{hkak} and to the recurrence relation \eqref{an}, in the context of the diffusion limit. This is done with the help of the pruned lookdown ancestral selection graph (pruned LD-ASG). Following the same lines, we obtain analogue interpretations in the finite case.  

We fix $M\subset[N]$ and $\tau>0$, and we consider a realisation of $\As_{[0,\tau]}^{M,N}(*)$. Let $0<\tau_1<\cdots<\tau_n<\tau$ be the corresponding coalescent, branching and collision times, and set $\tau_0:=0$, $\tau_{n+1}:=\tau$. The LD-ASG is obtained reordering the lines of $\As_{[0,\tau]}^{M,N}(*)$ at the times $\tau_k$. The correspondence between lines in the ASG and levels in the LD-ASG is given, for each $\beta\in[0,\tau]$, by a bijective function 
$$\pi_\beta:\As_\beta^{M,N}\rightarrow[K_{\beta}^{M,N}].$$
The function $\beta\in[0,\tau]\mapsto\pi_\beta$ remains constant on the intervals $[\tau_k,\tau_{k+1})$. If $M:=\{i_1,...,i_{|M|}\}$, with $i_1<\cdots<i_m$, we set $\pi_0(i_k):=k$. In addition, if we have constructed the LD-ASG in the interval $[0,\tau_k)$, with $k\leq n$, we extend its construction to $[\tau_k,\tau_{k+1})$ as follows (see Fig. \ref{fig8}):\\
$\bullet$ \textit{Coalescence}: if at time $\tau_k$ a neutral arrow goes from $i$ to $j$ in $\As_{[0,\tau]}^{M,N}(*)$, we draw a neutral arrow going from $\pi_{{\tau_k}^-}(i)\wedge\pi_{{\tau_k}^-}(j)$ to  $\pi_{{\tau_k}^-}(i)\vee\pi_{{\tau_k}^-}(j)$. The line $i$ is placed at level $\pi_{{\tau_k}^-}(i)\wedge\pi_{{\tau_k}^-}(j)$. The lines above $\pi_{{\tau_k}^-}(i)\vee\pi_{{\tau_k}^-}(j)$ are shifted one level downwards, and the others keep their positions, i.e. $\pi_{\tau_k}(i):= \pi_{{\tau_k}^-}(i)\wedge\pi_{{\tau_k}^-}(j)$, $\pi_{\tau_k}(\ell)=\pi_{{\tau_k}-}(\ell)-1$ if $\pi_{{\tau_k}^-}(\ell)>\pi_{{\tau_k}^-}(i)\vee\pi_{{\tau_k}^-}(j)$, and $\pi_{\tau_k}(\ell)=\pi_{{\tau_k}^-}(\ell)$ otherwise.\\
$\bullet$ \textit{Branching}: if at time $\tau_k$ the line $i$ branches into the lines $i$ and $j$ in $\As_{[0,\tau]}^{M,N}(*)$, an horizontal open arrowhead appears at level $\pi_{{\tau_k}^-}(i)$. The incoming branch $j$ emanates from it, and all the lines at levels $r\geq \pi_{{\tau_k}^-}(i)$ are shifted one level upwards, i.e. $\pi_{{\tau_k}}(j):=\pi_{{\tau_k}^-}(i)$, $\pi_{{\tau_k}}(\ell):=\pi_{{\tau_k}^-}(\ell)+1$ if $\pi_{{\tau_k}^-}(\ell)\geq\pi_{{\tau_k}^-}(i)$, and $\pi_{\tau_k}(\ell):=\pi_{{\tau_k}^-}(\ell)$ otherwise.\\
$\bullet$ \textit{Collision}: if at time ${\tau_k}$ the line $i$ collides the line $j$ in $\As_{[0,\tau]}^{M,N}(*)$, and $\pi_{{\tau_k}^-}(i)<\pi_{{\tau_k}^-}(j)$, we set $\pi_{{\tau_k}}:=\pi_{{\tau_k}^-}$, and we draw a selective arrow from $\pi_{{\tau_k}}(i)$ to $\pi_{{\tau_k}}(j)$.\\
$\bullet$ \textit{Exchange-collision}: if at time ${\tau_k}$ the line $i$ collides  with the line $j$ in $\As_{[0,\tau]}^{M,N}(*)$, and $\pi_{{\tau_k}^-}(i)>\pi_{{\tau_k}^-}(j)$, the line $i$ takes the position of line $j$, all the lines at levels $\pi_{{\tau_k}^-}(j)\leq r<\pi_{{\tau_k}^-}(i)$ are shifted one level upwards, and the positions of the other lines remain unchanged , i.e $\pi_{{\tau_k}}(i):=\pi_{{\tau_k}^-}(j)$, $\pi_{{\tau_k}}(\ell):=\pi_{{\tau_k}^-}(\ell)+1$ if $\pi_{{\tau_k}-}(j)\leq \pi_{{\tau_k}^-}(\ell)<\pi_{{\tau_k}^-}(i)$, and $\pi_{{\tau_k}}(\ell):=\pi_{{\tau_k}^-}(\ell)$ otherwise. In addition, an horizontal open arrowhead appears at level $\pi_{\tau_k}(i)$.\\
$\bullet$ \textit{Mutations}: a mutation event in the ASG is pasted at the corresponding level in the LD-ASG.\\
\ifpdf
\begin{figure}[!ht]
\begin{picture}(0,0)%
\centerline{\resizebox*{10.5cm}{1.5cm}{\includegraphics{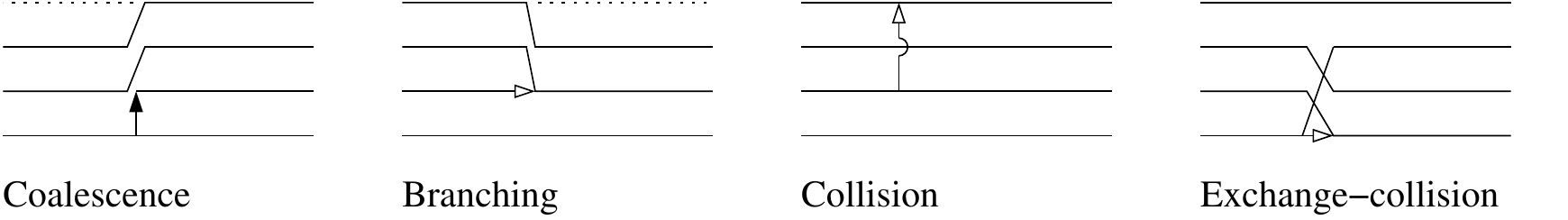}}}
\end{picture}%
\setlength{\unitlength}{4144sp}%
\begingroup\makeatletter\ifx\SetFigFont\undefined%
\gdef\SetFigFont#1#2#3#4#5{%
  \reset@font\fontsize{#1}{#2pt}%
  \fontfamily{#3}\fontseries{#4}\fontshape{#5}%
  \selectfont}%
\fi\endgroup%
\centerline{\resizebox*{10.5cm}{1.5cm}{
\begin{picture}(7951,1111)(1786,-935)
\end{picture}
}}
\caption{\label{fig8} Coalescence, branching and collisions in the LD-ASG}
\end{figure}
\fi
The resulting object is called the \textit{lookdown ancestral selection graph} in $[0,\tau]$ of the sample $M$ (see Fig. \ref{fig9}). We denote it by $L_{[0,\tau]}^{M,N}$. 
\ifpdf
\begin{figure}[!ht]
\begin{picture}(0,0)%
\centerline{\resizebox*{7.8cm}{3.5cm}{\includegraphics{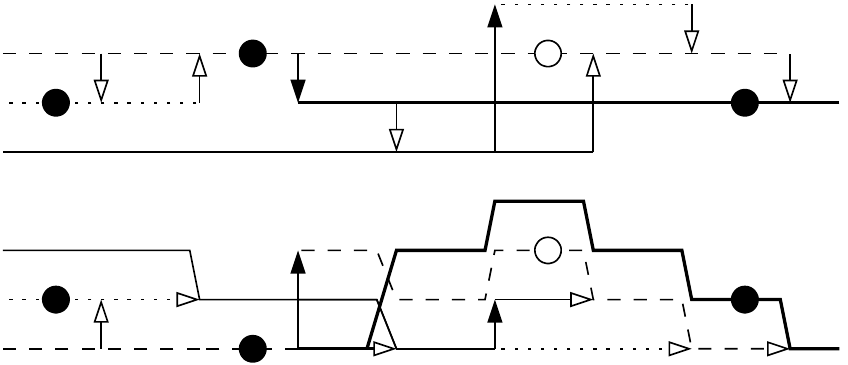}}}
\end{picture}%
\setlength{\unitlength}{4144sp}%
\begingroup\makeatletter\ifx\SetFigFont\undefined%
\gdef\SetFigFont#1#2#3#4#5{%
  \reset@font\fontsize{#1}{#2pt}%
  \fontfamily{#3}\fontseries{#4}\fontshape{#5}%
  \selectfont}%
\fi\endgroup%
\centerline{\resizebox*{7.8cm}{3.5cm}{
\begin{picture}(3859,1659)(2464,-1258)
\end{picture}%
}}
\caption{\label{fig9} The LD-ASG (below) corresponding to the untyped ASG (above). Time runs from right to left.}
\end{figure}
\fi

In what follows, we focus on the case $M=\{1\}$. In this case, $\As_{[0,\tau]}^{\{1\},N}$ has exactly one immune path, which can be identified in $L_{[0,\tau]}^{\{1\},N}$. However, we denote by \textit{immune line} a slightly different object. The immune line in the LD-ASG is the line which is at any time the ancestral line if all the lines at that time are of type one. 

Now, we proceed to prune the LD-ASG. We read off $L_{[0,\tau]}^{\{1\},N}$ from time $\beta=0$ to $\beta=\tau$ (from left to right in Fig. \ref{fig9}) using the following rules (see Fig. \ref{fig10}). If we encounter a mutation to type $0$ at the immune line, we don't do anything. If we encounter a mutation to type $0$ at an occupied level $i$ different from the immune line (i.e. $i$ is till now an almost neutral path), we insert at this time a barrier from level $i$ till level $N$, and we kill all the lines above the level $i$. If we meet a mutation to type $1$ at an occupied level $i$ different to the immune line (i.e. $i$ is a fictitious path), we kill the line $i$ and we shift all the lines above one level downwards. If we meet a mutation to type $1$ at the immune line, we relocate the immune line to the currently highest occupied level, and all the lines which were above the immune line are shifted one level downwards. The resulting object is called the \textit{pruned LD-ASG} and denoted by $\Ls_{[0,\tau]}^{N}$. Note that the pruning procedure can transform collisions or exchange-collisions in branching events.
\ifpdf
\begin{figure}[!h]
\begin{picture}(0,0)%
\centerline{\resizebox*{10.5cm}{1.1cm}{\includegraphics{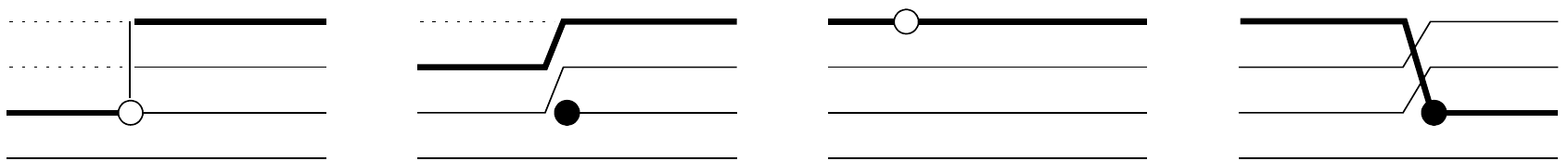}}}
\end{picture}%
\setlength{\unitlength}{4144sp}%
\begingroup\makeatletter\ifx\SetFigFont\undefined%
\gdef\SetFigFont#1#2#3#4#5{%
  \reset@font\fontsize{#1}{#2pt}%
  \fontfamily{#3}\fontseries{#4}\fontshape{#5}%
  \selectfont}%
\fi\endgroup%
\centerline{
\begin{picture}(7716,780)(193,-298)
\end{picture}
}
\caption{\label{fig10} Pruning procedure in the LD-ASG. The effect of the mutations differs in the immune line, represented in the figure in bold.}
\end{figure}
\begin{figure}[!ht]
\begin{picture}(0,0)%
\centerline{\resizebox*{7.8cm}{1.6cm}{\includegraphics{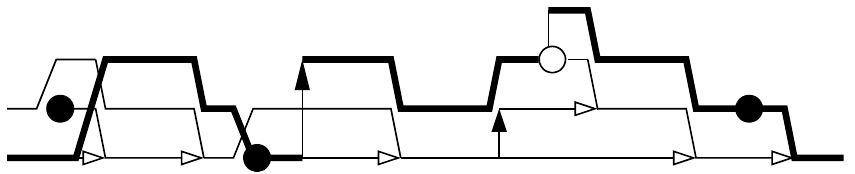}}}
\end{picture}%
\setlength{\unitlength}{4144sp}%
\begingroup\makeatletter\ifx\SetFigFont\undefined%
\gdef\SetFigFont#1#2#3#4#5{%
  \reset@font\fontsize{#1}{#2pt}%
  \fontfamily{#3}\fontseries{#4}\fontshape{#5}%
  \selectfont}%
\fi\endgroup
\centerline{
\begin{picture}(3891,780)(1318,-1033)
\end{picture}%
}
\caption{\label{fig11} Pruned LD-ASG corresponding to Fig. \ref{fig9}. The bold line represents the immune line.}
\end{figure}
\fi
\begin{proposition}\label{tal} The level of the ancestral line at time $\tau$ in $\Ls_{[0,\tau]}^{N}$ is either the lowest level of type $0$ at time $\tau$ or, the level of the immune line if all the lines are of type $1$ at time $\tau$. In particular, the ancestral line is of type $1$ at time $\tau$ if and only if all the lines in $\Ls_{[0,\tau]}^{N}$ are of type 1 at time $\tau$.
\end{proposition}
\begin{proof}
If all the lines in $\Ls_{[0,\tau]}^{N}$ are of type $1$ at time $\tau$, the level of the ancestral line is by definition the level of the immune line. Now we assume that at least one of the lines in $\Ls_{[0,\tau]}^{N}$ is of type $0$ at time $\tau$, and we denote by $i_*$ the lowest type-$0$ level. We have to show that $i_*$ is also the level of the ancestral line at time $\tau$.\\
Let $\sigma_1>0$ be the first branching time and $\sigma_1<\sigma_2<\cdots<\sigma_m<\sigma$ the consecutive coalescence, branching, collision, exchange-collision and mutation times in $\Ls_{[0,\tau]}^{N}$. We set $\sigma_{m+1}:=\tau$. Along the proof, $(i,\beta)\in[N]\times[0,\tau]$ represents the individual at level $i$ at time $\beta$ in $\Ls_{[0,\tau]}^{N}$. If there is at least one line of type $0$ in $(\sigma_k,\sigma_{k+1})$, we denote by $i_k$ the level of the lowest type-$0$ level.

We claim that, if $k\in[m]$ is such that there is at least one line of type $0$ in $(\sigma_k,\sigma_{k+1})$, and $(i_k,\tau_k)$ is a descendent of $(i_*,\tau)$, then either $(i_k,\tau_k)$ is over the immune line, mutates to type $1$ and all the lines are of type $1$ in $(\sigma_{k-1},\sigma_k)$, or $(i_{k-1},\tau_{k-1})$ is a descendant of $(i_*,\tau)$. If this is true, then starting with $k=m+1$ and iterating the claim, we deduce that the unique individual in $(0,\sigma_1)$ is a descendant of $(i_*,\tau)$, which proves the result.

We split the proof of the claim depending on the kind of event occurring at time $\tau_k$. If at time $\tau_k$ there is a mutation to type $1$ in one of the lines present in $(\sigma_k,\sigma_{k+1})$, we have three possibilities: (1) the mutation occurs at a level different to $i_k$, then the individual $(i_k,\tau_k)$ is placed at, the maybe different, level $i_{k-1}$  in $(\sigma_{k-1},\sigma_{k})$, (2) the mutation occurs at level $i_k$ and $i_k$ is not the level of the immune line, and the same conclusion holds, and (3) the mutation is at level $i_k$, which is also the level of the immune line, then $i_k$ is the highest occupied level in $(\sigma_k,\sigma_{k+1})$, and therefore in $(\sigma_{k-1},\sigma_{k})$ there are only type-$1$ lines. In the three cases the claim follows. If $\tau_k$ represents a mutation to type $0$, then by construction it is at the highest occupied level in $(\sigma_k,\sigma_{k+1})$, the individual $(i_k,\tau_k)$ remains at the same level in $(\sigma_{k-1},\sigma_{k})$, $i_k=i_{k-1}$ and the claim follows in this case. If $\tau_k$ is a branching, coalescence, collision or exchange-collision time, then by construction $(i_k,\tau_k)$ is the ancestor of $(i_{k-1},\tau_{k-1})$, and the proof of the claim is completed.
\end{proof}
Following a similar procedure as in Section \ref{MASG}, we construct in a Markovian way the pruned LD-ASG together with the level of its immune line. More precisely, we construct a Markov process of the form  $\mathbb{L}^{N}:=(\Ls_\beta^N,\Gamma^{N}_\beta, \Is_\beta^N)_{\beta\geq 0}$, where\\
$\bullet$ $\Ls_\beta^{N}\in[N]$ represents the number of occupied levels at time $\beta$ in the LD-ASG.\\
$\bullet$ $\Gamma^{N}:=\{\Gamma_{i}^{0,N},\Gamma_{i}^{1,N},\Gamma_{i}^{\vartriangle,N},\{\Gamma_{i,j}^{c,N},\Gamma_{i,j}^{e,N},\Gamma_{i,j}^{\blacktriangle,N}\}_{j>i} \}_{i\in[N]}$ is a collection of counting process encoding the mutation, branching, collision, exchange-collision and coalescence events.\\
$\bullet$ $\Is_\beta^{N}\in[N]$ represents the level of the immune line at time $\beta$ in the LD-ASG.\\
The graphical representation of the process $\mathbb{L}^{N}$ is obtained by drawing horizontal lines at all the occupied lines and interpreting the mutation, branching, collision, exchange-collision and coalescence events exactly as in the previous construction of the LD-ASG. Due to the exchangeability of the lines, we can read off the dynamics of the pruned LD-ASG from the configuration of arrows and circles given by the reproduction-mutation process $\Lambda^N:=(\Lambda^N_t)_{t\in\Rb}$. We start with $\Ls_0^N:=1$, $\Gamma_0^N:=0$ and $\Is_0^N:=1$ and we proceed as follows:\\
$\bullet$ if $t=-\beta$ is a jump of the process $\lambda_{i,j}^{\vartriangle,N}$ and $j\leq \Ls_{\beta^-}^N$, we have three options:
\begin{itemize}
 \item[$\star$] if $i>\Ls_{\beta^-}^N$: a branching occurs at level $j$. We set $\Ls_{\beta}^N=\Ls_{\beta^-}^N+1$ and $\Gamma_{j}^{\vartriangle,N}(\beta)=\Gamma_{j}^{\vartriangle,N}(\beta^-)+1$. In addition, if $j\leq\Is_{\beta^-}^N$, we set $\Is_{\beta}^N=\Is_{\beta^-}^N+1$, otherwise we set $\Is_{\beta}^N=\Is_{\beta^-}^N$.
  \item[$\star$] if $i<j$, a collision occurs between $i$ and $j$. We set $\Ls_{\beta}^N=\Ls_{\beta^-}^N$, $\Gamma_{i,j}^{c,N}(\beta)=\Gamma_{i,j}^{c,N}(\beta^-)+1$, and $\Is_{\beta}^N=\Is_{\beta^-}^N$.
  \item[$\star$] if $j<i$, an exchange collision occurs between $i$ and $j$. We set $\Ls_{\beta}^N=\Ls_{\beta^-}^N$ and $\Gamma_{j,i}^{e,N}(\beta)=\Gamma_{j,i}^{e,N}(\beta^-)+1$. If $\Is_{\beta^-}^N=i$, we set $\Is_{\beta}^N=j$. If $\Is_{\beta^-}^N=j$, then $\Is_{\beta}^N=j+1$. Otherwise, we set $\Is_{\beta}^N=\Is_{\beta^-}^N$.
\end{itemize}
$\bullet$ If $t=-\beta$ is a jump of $\lambda_{i,j}^{\blacktriangle,N}$ and $i\vee j\leq \Ls_{\beta^-}^N$: a coalescence occurs between the levels $i$ and $j$. Denoting, $k=i\wedge j$ and $\ell=i\vee j$, we set $\Ls_{\beta}^N=\Ls_{\beta^-}^N-1$ and  $\Gamma_{k,\ell}^{\blacktriangle,N}(\beta)=\Gamma_{k,\ell}^{\blacktriangle,N}(\beta^-)+1$. If $\Is_{\beta^-}^N>\ell$, then $\Is_{\beta}^N=\Is_{\beta^-}^N-1$. If $\Is_{\beta^-}^N=\ell$, then $\Is_{\beta}^N=k$. Otherwise, we set $\Is_{\beta}^N=\Is_{\beta^-}^N$.\\
$\bullet$ if $t=-\beta$ is a jump of $\lambda_{i}^{0,N}$ and $i\leq \Ls_{\beta^-}^N$: a mutation to type $0$ occurs at level $i$. We set $\Gamma_{k,\ell}^{0,N}(\beta)=\Gamma_{k,\ell}^{0,N}(\beta^-)+1$. If $\Is_{\beta^-}^N=i$, then $\Is_{\beta}^N=\Is_{\beta^-}^N$ and $\Ls_{\beta}^N=\Ls_{\beta^-}^N$. If $\Is_{\beta^-}^N\neq i$, then $\Ls_{\beta}^N=i$, and $\Is_{\beta}^N=i$ if $\Is_{\beta^-}^N>i$ or $\Is_{\beta}^N=\Is_{\beta^-}^N$ otherwise.\\
$\bullet$ if $t=-\beta$ is a jump of $\lambda_{i}^{1,N}$ and $i\leq \Ls_{\beta^-}^N$: a mutation to type $1$ occurs at level $i$. We set $\Gamma_{k,\ell}^{1,N}(\beta)=\Gamma_{k,\ell}^{1,N}(\beta^-)+1$. If $\Is_{\beta^-}^N=i$, then $\Is_{\beta}^N=L_{\beta^-}^N$ and $\Ls_{\beta}^N=\Ls_{\beta^-}^N$. If $\Is_{\beta^-}^N\neq i$, then $\Ls_{\beta}^N=\Ls_{\beta^-}^N-1$ and $\Is_{\beta}^N=\Is_{\beta^-}^N-1$ if $\Is_{\beta^-}^N>i$ or $\Is_{\beta}^N=\Is_{\beta^-}^N$ otherwise.\\
$\bullet$ other jumps in $\Lambda^N$ are ignored.

From construction, the line-counting process $\Ls^N=(\Ls_\beta^N)_{\beta\geq 0}$ is a continuous-time Markov chain with state space $[N]$ and transition rates given by
\begin{equation}\label{rqN}
 q_{\Ls^N}(i,j):=\left\{ \begin{array}{ll}
                             i\,(N-i)\,\,N^{-1}\,s& \textrm{if $j=i+1$},\\
       
                             i(i-1)N^{-1} +(i-1) \,u\,\nu_1+u\nu_0& \textrm{if $j=i-1$},\\
                   
                             u\nu_0& \textrm{if $j\in\{1,\dots,i-2\}$},\\
                             
                             i\left((N-i)s+(i-1)\right)N^{-1}+(i-1)\,u& \textrm{if $j=i$},\\
                             0 & \textrm{otherwise},
                         \end{array}\right.
\end{equation}
for $i,j\in[N]$. Since $\Ls^N$ is irreducible and its state space is finite, it has a unique stationary distribution, denoted by $\rho^N$. Let $\Ls_\infty^N$ be a random variable distributed as $\rho^N$. The next result gives a new graphical interpretation to the expression \eqref{hkak}.
\begin{proposition}\label{sran}
We have 
$$h_k^N:=\frac{k}{N}\sum\limits_{n=0}^{N-k}P(\Ls_\infty^{N}>n) \prod\limits_{j=0}^{n-1}\frac{N-k-j}{N-1-j}.$$
In addition, if we define, for each $n\in[N]_0$, $a_n^N:=P(\Ls_\infty^{N}>n)$, then $(a_n^N)_{n=0}^N$ satisfies \eqref{an} for all $n\in[N]_0$.
\end{proposition}
\begin{proof}
The first part of the statement follows from Proposition \ref{tal} using similar arguments as in Lemma \ref{rasghk}. 

From definition $P(\Ls_\infty^N>0)=1$ and hence \eqref{a0} is satisfied. Using the first statement for $k=N-1$ we deduce that \eqref{a1} holds. Furthermore, since $\rho^N$ is the stationary distribution of $\Ls^N$, we have $\rho^N Q_{\Ls^N}=0$, where $Q_{\Ls^N}$ is the generator of $\Ls^N$.
In particular, for $n\in[N-2]_2$, the $n$-th cequation in $\rho^N Q_{\Ls^N}=0$ reads
\begin{align*}
\rho_{n-1}^N&\left[\frac{(n-1)(N-n+1)}{N}\,s\right]+\rho_{n+1}^N \left[\frac{n(n+1)}{N}+nu\nu_1+u\nu_0\right]+u\nu_0\sum_{j=n+2}^{N} \rho_{j}^N\\
&= \rho_{n}^N\left[\frac{n(n-1)}{N}+(n-1)u\nu_1+u\nu_0+\frac{n(N-n)}{N}\,s+u\nu_0(n-2)\right].
\end{align*}
Reordering the terms we obtain
\begin{align}\label{tpk}
(n-1)&\left(\rho_{n-1}^N\left[\frac{N-n+1}{N}\,s\right]-u\nu_0 a_{n}^N-\rho_{n}^N \left[\frac{n}{N}+u\right]\right)\nonumber\\
&=n\left(\rho_{n}^N\left[\frac{N-n}{N}\,s\right]-u\nu_0 a_{n+1}^N-\rho_{n+1}^N \left[\frac{n+1}{N}+u\right]\right).
\end{align}
In a similar way, we derive from the first equation in $\rho^N Q_{\Ls^N}=0$ that
\begin{equation}\label{tp1}
\rho_{1}^N\left[\frac{N-1}{N}\,s\right]-u\nu_0 a_{2}^N-\rho_{2}^N \left[\frac{2}{N}+u\right]=0.
\end{equation}
Using, \eqref{tpk} and \eqref{tp1}, we deduce, for all $n\in[N-2]$, that
\begin{equation*}
 \rho_{n}^N\left[\frac{N-n}{N}\,s\right]-u\nu_0 a_{n+1}^N-\rho_{n+1}^N \left[\frac{n+1}{N}+u\right]=0.
\end{equation*}
It is straightforward to see that the previous equation is equivalent to \eqref{an}. It remains only to prove that the missing equation \eqref{me} holds. The latter is easily obtained from the last equation in $\rho^N Q_{\Ls^N}=0$, which is given by 
$$\rho_{N-1}^N\frac{s}{N}=\rho_N^N(1+u).$$
\end{proof}

\section{The asymptotic common ancestor type distribution and the asymptotic pruned LD-ASG}\label{s6}
In this section we aim to give a probabilistic interpretation to the function $h$ appearing in Theorem \ref{chk}. To do so, we construct, in the setting of the deterministic limit, an asymptotic version of the pruned LD-ASG.

Let us first study the asymptotic behaviour of $\Ls^N$. It is straightforward to see that, for any $i,j\in[N]_0$, 
$q_{\Ls^N}(i,j)\xrightarrow[N\rightarrow\infty]{}q_{\Ls}(i,j),$
where $q_{\Ls}$ is defined by
\begin{equation}\label{rq}
 q_{\Ls}(i,j):=\left\{ \begin{array}{ll}
                             i\,s& \textrm{if $j=i+1$},\\
                      
                             (i-1) \,u\,\nu_1+u\nu_0& \textrm{if $j=i-1$},\\
                           
                             u\nu_0& \textrm{if $j\in\{1,\dots,n-2\}$},\\
                             u-(s+u)i  & \textrm{if $j=i$},\\
                             0 & \textrm{otherwise}.
                         \end{array}\right.
\end{equation}
Let $\Ls=(\Ls_\beta)_{\beta\geq 0}$ be the continuous-time Markov chain corresponding to the transition rates given in \eqref{rq} starting at $\Ls_0:=1$. We refer to $\Ls$ as the asymptotic line-counting process.
\begin{lemma}\label{nexp}
 The process $\Ls$ is non-explosive.
\end{lemma}
\begin{proof}
From \cite[Proposition 8.7.2]{To92} the non-explosive condition is equivalent to
 $$\sum\limits_{k=0}^\infty \frac1{q_\Ls(\Ls_{S_k})}=\infty\quad \textrm{a.s.},$$
 where $S_0<S_1<\cdots$ denote the jump times of $\Ls$ and $q_\Ls(n):=-q_\Ls(n,n)$. In addition, we have $\Ls_{S_k}\leq k+2$. Therefore,
  $$\sum\limits_{k=0}^\infty \frac1{q_\Ls(\Ls_{S_k})}=\sum\limits_{k=0}^\infty \frac1{(u+s)\Ls_{S_k}-u}\geq \sum\limits_{k=0}^\infty \frac1{(u+s)(k+2)-u}=\infty\quad \textrm{a.s.}.$$
  This concludes the proof.
\end{proof}
The next proposition formalises the convergence of $\Ls^N$ to $\Ls$.
\begin{proposition}\label{calado}
The sequence of line-counting processes $(\Ls^N)_{N\geq 1}$ converges in distribution to the line-counting process $\Ls$. 
\end{proposition}
\begin{proof}
We denote by $(\Di,d_\infty^\circ)$ the space of c\`{a}dl\`{a}g functions on $[0,\infty)$ with values on $\Nb$ equiped with the metric $d_\infty^\circ$ defined in Appendix \ref{A1}, i.e. with the Skorohod topology.
 We have to show that, for all uniformly continuous and bounded function $F:\Di\rightarrow\Rb$, $\lim_{N\rightarrow\infty}E[F(\Ls^N)]=E[F(\Ls)]$. Let $F$ be such a function and fix $k\in\Nb$. Note that for every $N\geq k$
 \begin{align}\label{splite}
\left|E\left[F(\Ls^N)\right]-E\left[F(\Ls)\right]\right|&\leq E\left[\left|F(\Ls^N)-F\left(\Ls_{\cdot\wedge T_k(\Ls^N)}^N\right)\right|\right]\nonumber\\
 &+\left|E\left[F\left(\Ls_{\cdot\wedge T_k(\Ls^N)}\right)\right]-E\left[F\left(\Ls_{\cdot\wedge T_k(\Ls)}\right)\right]\right|\nonumber\\
 &+E\left[\left|F\left(\Ls_{\cdot\wedge T_k(\Ls)}\right)-F(\Ls)\right|\right],
 \end{align}
 where $T_k$ is the function defined in Appendix \ref{A1}. Note that $T_k(\Ls^N)$ and $T_k(\Ls)$ are a.s. finite. We denote by $E_{N,k}^1$, $E_{N,k}^2$ and $E_{k}^3$ respectively the first, second and third term on the right-hand side of \eqref{splite}.
 
 The processes $\Ls^{N}(k):=(\Ls_{t\wedge T_k(\Ls^N)}^N)_{t\geq 0}$ and $\Ls(k):=(\Ls_{t\wedge T_k(\Ls)})_{t\geq 0}$ are continuous-time Markov chains with space state $[k]$. Moreover, it is straightforward to see that the transition rates of  $\Ls^{k,N}$ converge to the transition rates of $\Ls^k$. This in turn implies, for every real-valued, bounded function on $[k]$, that
$$\lim\limits_{N\rightarrow\infty}\sup\limits_{i\in[k]}|A_{\Ls^N(k)}f(i)-A_{\Ls(k)}f(i)|=0,$$
where $A_{\Ls^N(k)}$ and $A_{\Ls(k)}$ denote the generators of $\Ls^N(k)$ and $\Ls(k)$. From \cite[Theorems 1.6.1 and 4.2.11]{Ku86} we get that $\Ls^N(k)$ converges in distribution to $\Ls(k)$. Thus,
\begin{equation}\label{ek2}
 \lim\limits_{N\rightarrow\infty}E_{N,k}^2=0.
\end{equation}
Fix $\varepsilon\in(0,1)$. Since, $F$ is uniformly continuous, there is $\delta_F(\varepsilon)\in(0,\varepsilon)$ such that
$$d_\infty^\circ(w,w_*)\leq \delta(\varepsilon)\Rightarrow |F(w)-F(w_*)|\leq \varepsilon.$$
As a consequence of this and Lemma \ref{trunc}, we deduce that
\begin{equation}\label{ekn1}
 E_{k,N}^1\leq \varepsilon + 2||F||_\infty P(T_k(\Ls^N)\leq n_F(\varepsilon)),
\end{equation}
where $n_F(\varepsilon):=\lfloor\log_2(\frac{1}{\delta(\varepsilon)})\rfloor+1$. Similarly, we get
\begin{equation}\label{ek3}
 E_{k}^3\leq \varepsilon + 2||F||_\infty P(T_k(\Ls)\leq n_F(\varepsilon)).
\end{equation}
Note that $T_k(\Ls^N)=T_k(\Ls^N(k))$ and $T_k(\Ls)=T_k(\Ls(k))$. Since $\Ls^N(k)$ converges in distribution to $\Ls(k)$ and the function $T_k$ is continuous (see Lemma \ref{ctk}), we deduce from the mapping theorem (see \cite[Theorem 2.7]{Bi99}) that $T_k(\Ls^N)$ converges in distribution to $T_k(\Ls)$. In particular, from \eqref{ekn1}, we get
\begin{equation}\label{ek1}
 \limsup\limits_{N\rightarrow\infty}E_{k,N}^1\leq \varepsilon + 2||F||_\infty P(T_k(\Ls)\leq n_F(\varepsilon)).
\end{equation}
Using \eqref{ek2},\eqref{ek3} and \eqref{ek1}, we obtain
$$ \limsup\limits_{N\rightarrow\infty}\left|E\left[F(\Ls^N)\right]-E\left[F(\Ls)\right]\right|\leq 2\varepsilon +4||F||_\infty P(T_k(\Ls)\leq n_F(\varepsilon)).$$
Lemma \ref{nexp} implies that the last term converge to zero when $k$ tends to $\infty$. Summarizing, we have proven that, for all $\varepsilon\in(0,1)$,
$$\limsup\limits_{N\rightarrow\infty}\left|E\left[F(\Ls^N)\right]-E\left[F(\Ls)\right]\right|\leq 2\varepsilon.$$
The result follows.
\end{proof}

\begin{lemma}\label{lsd}
The process $\Ls$ has a unique stationary distribution, which is given by the geometric distribution of parameter $1-\ell_-$.
\end{lemma}
\begin{proof}
If $\rho:=(\rho_k)_{k\geq 1}$ denotes the geometric distribution of parameter $1-\ell_-$, one can easily check that $\rho$ is a stationary distribution for $Q_\Ls$, i.e $\rho\,Q_\Ls=0$. Since $Q_\Ls$ is irreducible and non-explosive, we deduce from \cite[Theorem 3.5.3]{No98} that $Q_\Ls$ is positive recurrent. Therefore, the uniqueness of the stationary distribution follows from \cite[Theorem 3.5.2]{No98}.
\end{proof}
Let $\Ls_\infty^N$ and $\Ls_\infty$ be random variables following the stationary distributions $\Ls^N$ and $\Ls$, respectively. Corollary \ref{convln} can be translated in terms of $\Ls_\infty^N$ and $\Ls_\infty$ as
\begin{equation}\label{convsd2}
 \Ls_\infty^N\xrightarrow[N\rightarrow\infty]{}\Ls_\infty.
\end{equation}
Using this and Proposition \ref{sran}, we recover Theorem \ref{chk} in the following form
$$\frac{k_N}{N}\xrightarrow[N\rightarrow\infty]{}x\in(0,1)\Rightarrow  h_{k_N}^N \xrightarrow[N\rightarrow\infty]{}\sum\limits_{n=0}^\infty P(\Ls_\infty>n)x(1-x)^n=h(x).$$
Now, we aim to construct an asymptotic version of the pruned LD-ASG having $\Ls$ as line-counting process. Note that the coalescence and collision rates in the finite ASG, and the LD-ASG, converge to zero when the population size tends to infinity. In particular, collisions and coalescence events will be absent in any suitable asymptotic version of the pruned LD-ASG. Thus, the notion of common ancestor does not make sense anymore. Nevertheless, we provide a nice interpretation to the function $h(x)$ at the end of this section.

We turn now to the construction of the asymptotic pruned LD-ASG. We point out first that the abscence of coalescence implies that the immune line will be always located at the highest occupied level.

We start with a realisation of $(\Ls_\beta)_{\beta\geq 0}$  started at $\Ls_0=1$. If $\tau_1>0$ is the first jumping time of $\Ls$, then the asymptotic pruned LD-ASG in $[0,\tau_1)$ consists of a single individual placed at level $1$. Moreover, the asymptotic pruned LD-ASG remains constant on intervals of the form $[\tau,\tau_*)$, where $\tau<\tau_*$ are two consecutive jumps of $\Ls$. It remains to describe the evolution at the jumping times of $\Ls$. To do so, assume we have constructed the asymptotic pruned LD-ASG in $[0,\tau)$ and that $\tau$ is a jumping time of $\Ls$
\begin{itemize}
 \item If $\Ls_\tau-\Ls_{\tau^-}=1$, then a branching event occurs. A star appears at a level $i$ chosen between the $\Ls_{\tau^-}$ current lines. The incoming branch emanates from the star and all the lines at levels $k\geq i$ are shifted one level upwards.
 \item If $\Ls_\tau-\Ls_{\tau^-}=-1$ and $\Ls_{\tau^-}=n$, with probability $\frac{u\nu_0}{u\nu_0+(n-1)u\nu_1}$ a mutation to type $0$ occurs at level $n-1$, we insert an infinite vertical barrier starting from level $n-1$ and we kill  all the lines above level $n-1$. If no mutation to type $0$ happens, then a mutation to type $1$ takes place at a level $i<n$ chosen at random, the corresponding line is killed and we shift all the lines above one level downwards.
 \item If $\Ls_\tau-\Ls_{\tau^-}<-1$ and $\Ls_{\tau}=n$, a mutation to type $0$ occurs at level $n$, we insert a vertical barrier from level $n$ till infinity and we kill all the lines above level $n$.
 \end{itemize}
 Now we introduce a new notion which plays the role of the common ancestor in the deteministic limit setting. We start at time $\beta=0$ with a generic individual in the population, and we trace back the type of its ancestor. Furthermore, the type of the ancestor at time $\beta$ of the chosen individual is denoted by $I_\beta^*$ and is called the \textit{representative ancestral type} at time $\beta$. 
 \begin{remark}
 In the finite case, after the first botleneck, the representative ancestral type coincides with the common ancestor type. 
 \end{remark}
\begin{proposition}\label{ptal}
The representative ancestral type is $1$ at time $\beta$ if and only if all the lines in the asymptotic pruned LD-ASG at time $\beta$ are of type 1.
\end{proposition}
\begin{proof}
The proof is analogous to the proof of Proposition \ref{tal}.
\end{proof}
Let $J_\beta^k\in\{0,1\}$ be the type that is assigned at time $\beta$ to the individual placed at level $k\in\{1,...,L_{\beta}\}$ in the asymptotic pruned LD-ASG. We assume that the asymptotic pruned LD-ASG is constructed under a probability $P_x$ such that the types are assigned in an i.i.d. manner with $P_x(J_k^\beta=0)=x$. The following result provides a probabilistic meaning to the function $h$. 
\begin{corollary}\label{chdl}
For all $x\in(0,\infty)$, $h(x)=\lim\limits_{\beta\rightarrow\infty}P_x(I_\beta^*=0).$
\end{corollary}
\begin{proof}
From Proposition \ref{ptal}, we see that:
\begin{align*}
&P_x(I_\beta^*=0)=\sum\limits_{n=0}^\infty P_x(J^{n+1}_\beta=0,\,J^k_\beta=1\, \forall k\leq n, \Ls_\beta> n)\\
&=\sum\limits_{n=0}^\infty P_x(J^{n+1}_\beta=0,\,J_\beta^k=1\, \forall k\leq n)\,P_x( \Ls_\beta> n)=\sum\limits_{n=0}^\infty x(1-x)^{n}P_x( \Ls_\beta> n).
\end{align*}
Taking the limit when $\beta$ tends to infinity in both sides and using \eqref{convsd2}, we get
$$\lim\limits_{\beta\rightarrow\infty}P_x(I_\beta^*=0)=\sum\limits_{n=0}^\infty x(1-x)^{n}P_x( \Ls_\infty> n).$$
The result follows from the definition of $h$ and Lemma \ref{lsd}.
\end{proof}
\begin{remark}
Assume that we are in the stationary regime, i.e. under $P_{x_0^+}$, then the representative ancestral type coincides with the population average of the ancestral types in a related branching model (see \cite[Theorem 3.1]{GeBa03}). This is formalised in Lemma \ref{ratad}.
\end{remark}

\appendix
\section{\texorpdfstring{Some remarks on the Skorohod topology in $\Db_{\Nb}[0,\infty)$}{}}\label{A1}
For each $t\in(0,\infty)$, we denote by $\Dt{t}:=\Db_{\Nb}[0,t]$ the space of c\`{a}dl\`{a}g functions on $[0,t]$ with values on $\Nb$. We start by recalling the Skorohod topology in $\Dt{t}$. 

Let $\Ct{t}$ be the class of strictly increasing, continuous functions from $[0,t]$ onto itself. For $\lambda\in\Ct{t}$, we set
$$||\lambda||^\circ:=\sup\limits_{0\leq u<s\leq t}\left|\log\left(\frac{\lambda(s)-\lambda(u)}{s-u}\right)\right|.$$
Now, we define the metric $d_t^\circ$ in $\Dt{t}$ as follows:
$$d_t^\circ(f,g):=\inf\limits_{\lambda\in\Ct{t}}\{||\lambda||^0\vee||f-g\circ\lambda||_{t,\infty}\},$$
where $||f||_{t,\infty}:=\sup_{s\in[0,t]}|f(s)|$. The metric $d_t^\circ$ gives the Skorohod topology in $\Dt{t}$.

Similarly, we denote by $\Di:=\Db_{\Nb}[0,\infty)$ the space of c\`{a}dl\`{a}g functions on $[0,\infty)$ with values on $\Nb$. Let $\Ct{\infty}$ be the class of increasing, continuous functions from $[0,\infty)$ onto itself.

Now, we define, for each $m\in\Nb$, the function $x_m$ by setting $x_m(s):=1$ for $0\leq s\leq m-1$, $x_m(s):=m-s$ for $m-1\leq s\leq m$, and $x_m(s):=0$ for $s\geq m$. Finally, we define the metric $d_\infty^\circ$ in $\Di$ as follows:
$$d_\infty^\circ(f,g):=\sum\limits_{m=1}^\infty \frac{1\wedge d_m^\circ(x_m\, f,x_m \,g)}{2^m}.$$
The metric $d_\infty^\circ$ gives the Skorohod topology in $\Di$. 
\begin{lemma}\label{trunc}
For all $w\in\Di$ and $t\in[0,\infty)$, we have $$d_\infty^\circ(w,w(\cdot\wedge t))\leq \frac{1}{2^{\lfloor t\rfloor}}.$$
\end{lemma}
\begin{proof}
 Note first that for all $m\leq t$, $d_m^\circ(x_m\, w,x_m\, w(\cdot\wedge t))=0$. Therefore, we have
$$d_\infty^\circ(w,w(\cdot\wedge n))=\sum\limits_{m=\lfloor t\rfloor+1}^\infty \frac{1\wedge d_m^\circ(x_m\, w,x_m \,w(\cdot\wedge t))}{2^m}\leq \sum\limits_{m=\lfloor t\rfloor+1}^\infty \frac{1}{2^m}=\frac{1}{2^{\lfloor t\rfloor}}.$$
\end{proof}
Now, given $k\in\Nb$ and $w\in\Di$, we define $T_k(w):=\inf\{s\geq 0: w(s)=k\}$ and $\Di(k):=\{w\in\Di:\, T_k(w)<\infty\}$.
\begin{lemma}\label{ctk}
The function $T_k:\Di(k)\rightarrow [0,\infty)$ is continuous.
\end{lemma}
\begin{proof}
Let us consider $\{w_n\}_{n\in\Nb}\subset \Di(k)$ and $w\in\Di(k)$ such that
$$d_\infty^\circ(w_n,w)\xrightarrow[n\rightarrow\infty]{}0.$$
We have to show that $T_k(w_n)\xrightarrow[n\rightarrow\infty]{}T_k(w)$. From \cite[Theorem 16.1]{Bi99}, there is a sequence $\{\lambda_n\}_{n\in\Nb}\subset\Ct{\infty}$ such that
$$\sup\limits_{t<\infty}|\lambda_n(t)-t|\xrightarrow[n\rightarrow\infty]{}0\quad\textrm{and}\quad\sup\limits_{t\leq m}|w_n(\lambda_n(t))-w(t)|\xrightarrow[n\rightarrow\infty]{}0,$$
for all $m\in\Nb$. Therefore, defining $m_k:=T_k(w)+1$, we deduce that for all $\varepsilon>0$, there is $n_0(\varepsilon)\in\Nb$, such that, for all $n\geq n_0(\varepsilon)$:
$$\sup\limits_{t<\infty}|\lambda_n(t)-t|\leq\varepsilon\quad\textrm{and}\quad\sup\limits_{t\leq m}|w_n(\lambda_n(t))-w(t)|\leq\varepsilon.$$
This implies for $\varepsilon\in(0,1)$ that for all $n\geq n_0(\varepsilon)$
$$\sup\limits_{t\leq m_k}|w_n(\lambda_n(t))-w(t)|=0.$$
Therefore, $w_n(\lambda_n(T_k(w)))=k$ and, for each $t<T_k(w)$, $w_n(\lambda_n(t))\neq k$. This implies that $\lambda_n(T_k(w))=T_k(w_n)$. Consequently, we have
$$|T_k(w_n)-T_k(w)|=|\lambda_n(T_k(w))-T_k(w)|\leq \varepsilon.$$
The continuity of $T_k$ follows.
\end{proof}
\section{A related 2-type branching process}\label{bm}
Assume that $u>0$ and $s>0$. We consider a population composed of individuals of two types, $0$ and $1$, evolving in the following way. Each individual of type $0$ waits for an exponential time with parameter $1+s+u\nu_1$, and then splits or mutates to type $1$ with probabilities $(1+s)/(1+s+u\nu_1)$ and $u\nu_1/(1+s+u\nu_1)$, respectively. On the other hand, an individual of type $1$ waits for an exponential time with parameter $1+u\nu_0$, and then splits or mutate to type $0$ with probabilities $1/(1+u\nu_0)$ and $u\nu_0/(1+u\nu_0)$, respectively. In other words, the population evolves according to a 2-type branching process in continuous time. We summarize here the asymptotic properties of this model (see e.g. \cite{BG07} and \cite{BB08}).

In this framework, consider the process $\widehat{Y}=(\widehat{Y}_t^0, \widehat{Y}_t^1)_{t\geq 0}$, where $\widehat{Y}_t^j$, $j\in\{0,1\}$ denotes the number of individuals of type $j$ at time $t$. Additionally, we denote by $E_i[\widehat{Y}_t^j]$ the expected number of $j$-individuals at time $t$ in a population started by a single $i$-individual at time $0$. It is well known that $E_i[\widehat{Y}_t^j]=\left(e^{At}\right)_{ij}$, where $A$ is the matrix 
\begin{equation}\label{ldeteq}
A:=\begin{pmatrix}
     1+s-u\nu_1&u\nu_1\\u\nu_0&1-u\nu_0\\
    \end{pmatrix}.
\end{equation}
The asymptotic properties of $\widehat{Y}$ are expressed in terms of the largest eigenvalue of the matrix $A$, $\lambda_+=1+sx_0^+$, and the corresponding, properly normalized, left and right eigenvectors. Additionally, if $\pi:=(\pi_0,\pi_1)$ denotes the left eigenvector associated to $\lambda_+$ normalized such that $\pi_0+\pi_1=1$, then $\pi_0:=x_0^+$ and $\pi_1:=1-x_0^+.$
The right eigenvector of $A$ associated to $\lambda_+$, $\bar{h}=(\bar{h}_0,\bar{h}_1)$, normalized such that $\bar{h}_0\pi_0+\bar{h}_1\pi_1=1$ is given by
$$\bar{h}_0=\frac{u\nu_0+sx_0^+}{u\nu_0+s{x_0^+}^2}\quad\textrm{and}\quad \bar{h}_1=\frac{u\nu_0}{u\nu_0+s{x_0^+}^2}.$$ 

We know from \eqref{at} that $x_0^+$ is the asymptotic proportion of individuals of type $0$ in the deterministic 2-type selection mutation model. The same interpretation for $x_0^+$ holds in the branching model, since
$$\lim\limits_{t\rightarrow\infty}\frac{\widehat{Y}_t^0}{\widehat{Y}_t^0+\widehat{Y}_t^1}=x_0^+,\quad\textrm{conditionally on survival}.$$
Moreover, we have
$$\lim\limits_{t\rightarrow\infty}\frac{\ln\left(\widehat{Y}_t^0+\widehat{Y}_t^1\right)}{t}=1+sx_0^+,\quad\textrm{conditionally on survival},$$
i.e. $\lambda_+$ is the asymptotic growth rate of the population. 
In addition, $\bar{h}_i$, for $i\in\{0,1\}$, measures the asymptotic mean offspring size of an $i$ individual, relative to the total size of the population. More precisely, we have
$$\lim\limits_{t\rightarrow\infty}E^i\left[\widehat{Y}_t^0+\widehat{Y}_t^1\right]\,e^{-\lambda_+ t}=\bar{h}_i.$$
Finally, the vector $\alpha:=(\alpha_0,\,\alpha_1)$ given by $\alpha_i:=\bar{h}_i\pi_i$, for $i\in\{0,1\}$, describes the population average of the ancestral types and is called the ancestral distribution. 
\begin{lemma}[representative ancestral type and ancestral distribution]\label{ratad}
We have
$$\bar{h}_0=\frac{\Delta+(s+u)\sqrt{\Delta}}{2\Delta}=\frac{h(x_0^+)}{x_0^+},$$
where $\Delta:=(s-u)^2+4u\nu_0s.$ In particular, we have $h(x_0^+)=\alpha_0$, i.e. the representative ancestral type distribution starting with a stationary configuration equals the ancestral distribution in the 2-type branching model described here.
 \end{lemma}
\begin{proof}
 From the definition of $x_0^+$ and $\Delta$, we have $2(u\nu_0+sx_0^+)=2u\nu_0+s-u+\sqrt{\Delta},$
 and that $2s(u\nu_0+s{x_0^+}^2)=\Delta+(s-u)\sqrt{\Delta}.$
Thus, we obtain
$$\bar{h}_0=\frac{s(2u\nu_0+s-u+\sqrt{\Delta})}{\Delta+(s-u)\sqrt{\Delta}}=\frac{s(2u\nu_0+s-u+\sqrt{\Delta})}{\Delta-(s-u)\sqrt{\Delta}}\times \frac{\Delta-(s-u)\sqrt{\Delta}}{\Delta-(s-u)\sqrt{\Delta}},$$
and the first identity follows after simplifications. In order to obtain the second identity, we note that, from the relation between $\ell_-$ and $x_0^+$,
$$2su\nu_1(1-\ell_-(1-x_0^+))=2s(2u\nu_1-u-s+(s+u)x_0^+)=-\Delta+(s+u)\sqrt{\Delta}.$$
As a consequence, we have
$$\frac{h(x_0^+)}{x_0^+}=\frac{2su\nu_1}{-\Delta+(s+u)\sqrt{\Delta}},$$
and the second identity follows using similar arguments as before.
\end{proof}
\subsection*{Acknowledgements}
I would like to thank Ellen Baake for bringing this problem to my attention and for enlightening the development of this work. I am also grateful to Tom Kurtz and Anton Wakolbinger for stimulating and fruitful discussions. This project received financial support from the Priority Programme \textit{Probabilistic Structures in Evolution} (SPP 1590), which is funded by Deutsche Forschungsgemeinschaft.
\bibliographystyle{acm}
\bibliography{reference}
\end{document}